%% file: main.tex
\documentclass[12pt]{article}

\include{header}

\usepackage{xcolor}

\title{A projected Nesterov-Kaczmarz approach to stellar population-kinematic distribution reconstruction in Extragalactic Archaeology}
\author{
Fabian Hinterer\footnote{Johannes Kepler University Linz, Institute of Industrial Mathematics, Altenbergerstra{\ss}e 69, A-4040 Linz, Austria, (fabian.hinterer@indmath.uni-linz.ac.at)} ,
Simon Hubmer\footnote{Johann Radon Institute Linz, Altenbergerstra{\ss}e 69, A-4040 Linz, Austria, (simon.hubmer@ricam.oeaw.ac.at), Corresponding author.} ,
Prashin Jethwa\footnote{University of Vienna, Department of Astrophysics, T{\"u}rkenschanzstra{\ss}e 17, A-1180 Vienna, Austria, (prashin.jethwa@univie.ac.at)} ,
\\
Kirk M. Soodhalter\footnote{Trinity College, School of Mathematics, Dublin 2, Ireland, (ksoodha@maths.tcd.ie)} ,
Glenn van de Ven\footnote{University of Vienna, Department of Astrophysics, T{\"u}rkenschanzstra{\ss}e 17, A-1180 Vienna, Austria, (glenn.vandeven@univie.ac.at)} , 
Ronny Ramlau\footnote{Johannes Kepler University Linz, Institute of Industrial Mathematics, Altenbergerstra{\ss}e 69, A-4040 Linz, Austria, (ronny.ramlau@jku.at)} \footnote{Johann Radon Institute Linz, Altenbergerstra{\ss}e 69, A-4040 Linz, Austria, (ronny.ramlau@ricam.oeaw.ac.at)} 
}

\begin{document}

\maketitle

\begin{abstract}

In this paper, we consider the problem of reconstructing a galaxy's stellar population-kinematic distribution function from optical integral field unit measurements. These quantities are connected via a high-dimensional integral equation. To solve this problem, we propose a projected Nesterov-Kaczmarz reconstruction (PNKR) method, which efficiently leverages the problem structure and incorporates physical prior information such as smoothness and non-negativity constraints. To test the performance of our reconstruction approach, we apply it to a dataset simulated from a known ground truth density, and validate it by comparing our recoveries to those obtained by the widely used pPXF software.

\end{abstract}



\section{Introduction}
Galactic archaeology aims to study the formation and evolution of galaxies from observations in the local Universe, i.e., as viewed in the present day \cite{Freeman_Bland_Hawthorn_02}. Information about a galaxy's past can be inferred from present-day observations since certain quantities are approximately conserved throughout galactic evolution. The quantities describe stellar-orbits, which can be inferred from measured stellar kinematics, and stellar-population properties such as age and chemistry. In the Milky Way for example, judicious selections in population-kinematic space have led to the discovery of an ancient, massive galaxy merger in the inner-halo of our Galaxy \cite{Belokurov18,helmi18}. For external galaxies, where individual stars cannot be resolved, the task of \emph{extragalactic} archaeology becomes more difficult.

The outskirts of galaxies are filled by stellar debris of ongoing mergers (see Figure \ref{fig_stream_survey}), and upcoming deep imaging surveys e.g. \cite{stellar_stream_legacy_survey} will detect such debris around thousands of nearby galaxies. Spectroscopic data is a powerful complement to imaging and offers access to stellar-population and stellar-kinematic information. Integral-field unit (IFU) datacubes combine the two types of data: they are 2D images with a third dimension of spectral wavelength per image pixel. Several surveys have obtained IFU datacubes for thousands of galaxies e.g. \cite{califa, sami2, Bundy15, Sarzi18} though these are typically limited to relatively small fields-of-view. These datasets could be treasure troves to detect ancient, massive merger remnants in the central regions of galaxies e.g. \cite{Zhu22}.

\begin{figure}[ht!]
    \centering
    \includegraphics[width=\textwidth]{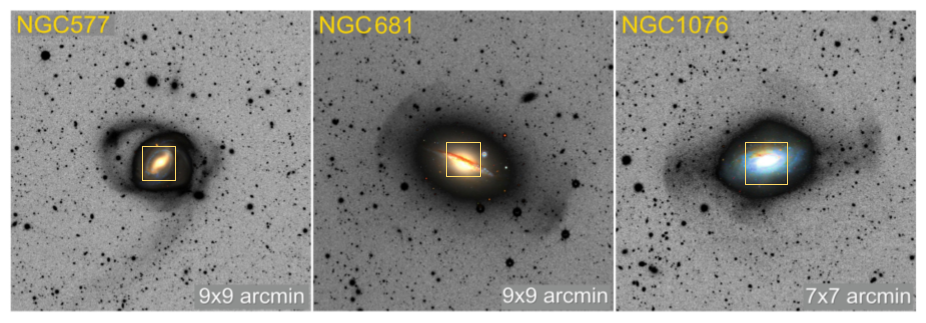}
    \caption{Three galaxies surrounded by the remnants of ongoing galactic mergers. The inset yellow square shows the smaller field of view of the MUSE IFU instrument on the Very Large Telescope (VLT) of the European Southern Observatory (ESO). Stellar population-kinematic modelling in the centers of galaxies may reveal ancient mergers. Images adapted from the Stellar Stream Legacy Survey \cite{stellar_stream_legacy_survey}.}
    \label{fig_stream_survey}
\end{figure}

Typical IFU data-analyses are not optimised for galactic archaeology. These analyses spatially bin the datacube into local spectra \cite{Cappellari03} which are modelled independently using spectral fitting tools e.g. \cite{Cappellari17}. By spatially binning the datacube, we limit what we can detect: global signals which are present when the datacube is considered altogether may be hidden in the noise when it is dissected into independent regions. Physically, we expect global signals to exist, since orbital dynamics act to disperse features over large areas. Our goal in this paper is to develop a technique for full datacube-modelling for stellar population-kinematic modelling of galaxies. 

Existing tools for 3D-datacube modelling of galaxies have limitations which make them unsuitable for galactic archaeology. To date, they all assume a disk model for the galaxy and/or fit data only for a single gas emission-line \cite{davis13, di_teodoro_15, Bouche_15, Bekiaris16, rizzo18, Varidel19}. Single line data is not sufficient for stellar modelling, where several absorption features must be simultaneously considered. Disk-models are not flexible enough for galactic archaeology, since satellite mergers can deposit stars on significantly non disk-like orbits. Though different disk models make very different assumptions - e.g. some based on a tilted-rings, or parametric or non-parametric rotation curves - none of these are able to capture the diversity of possible galactic orbits. For our proposed method, we therefore will use a non-parametric description of the stellar distribution function, instead promoting \emph{model-independent} physical prior-information such as spatial smoothness and non-negativity constraints in the recovery.

Mathematically, we will see that this reconstruction problem amounts to solving a very large scale inverse problem resulting from an integral equation with a specific system structure. In this paper, we mathematically analyse this problem and  propose a projected Nesterov-Kaczmarz reconstruction (PNKR) method for its solution. The PNKR method leverages the system structure and incorporates the available physical prior information, resulting in an accurate and numerical efficient algorithm. We demonstrate the performance of the PNKR method on numerical tests, including a comparison of our recoveries to those obtained by the widely used pPXF software \cite{Cappellari17}.
 
The outline of this paper is as follows: In Section~\ref{sect_background}, we review the necessary astrophysical background of galactic archaeology, and in Section~\ref{sect_model} derive a mathematical model for the corresponding density reconstruction problem. In Section~\ref{sect_reconstruction}, we then introduce our PNKR method for the efficient solution of this problem, which we test and validate numerically in Section~\ref{sect_numerical_examples}, followed by a short conclusion in Section~\ref{sect_conclusion}.

\section{Astrophysical background}\label{sect_background}

In this section, we derive a mathematical model for galactic archaeology. Our aim is to infer the galaxy's stellar population distribution function $f = f(\xv,v,z,t)$ from optical integral field spectroscopy measurements $y = y(\xv,\lambda)$. As indicated, these functions depend on the following variables:
    \begin{itemize}
        \item $\xv = (x_1,x_2)$ denoting 2D positional coordinates,
        \item $v$ denoting the velocity along the line-of-sight of stars in the galaxy,
        \item $z$ denoting the metal content of the stellar population, also known as metallicity,
        \item $t$ denoting the age of the stellar population,
        \item $\lambda$ denoting the wavelength of measured light.
    \end{itemize}
The target function $f$ therefore encodes the distribution of galaxy stars in the 5D space of (2D) position, velocity, age and metallicity. IFU data are obtained from telescope observations, where light from a 2D image is split into constituent wavelengths. Though different instrument designs achieve this goal in different ways (e.g.\ using microlens arrays, fiber bundles, or image splitters) standard instrument pipelines convert the raw-data into the 3D datacube $y(\xv,\lambda)$, which we take as the input for our problem.

The distribution function $f$ and the data $y$ are connected via templates of single-stellar population (SSP) spectra $S = S(\lambda;z,t)$. These are tabulated functions which describe the rest-wavelength spectrum of a stellar population, i.e. the flux of light per wavelength interval of a population of stars with metal content $z$ and age $t$, which is stationary with respect to the observer. The template spectra $S(\lambda;z,t)$ are determined empirically from the spectra of nearby stars and isochrone models, and hence are known to a high degree of accuracy. Figure~\ref{fig_something_s} shows some examples of the MILES SSP models \cite{Vazdekis_SanchezBlazquez_FalconBarroso_Cenarro_Beasley_Cardiel_Gorgas_Peletier_2010}, which we adopt in this work. The distribution function $f$ itself is normalised such that the total stellar mass of the galaxy is given by
    \begin{equation}\label{eq_mass_normalisation}  
        M^*_\text{total} = \int \int \int \int f(x, v, z, t) \, dx \, dv \, dz \, dt \,,
    \end{equation}
and $f$ must be everywhere non-negative.

\begin{figure}[ht!]
    \centering
    \includegraphics[width=\textwidth]{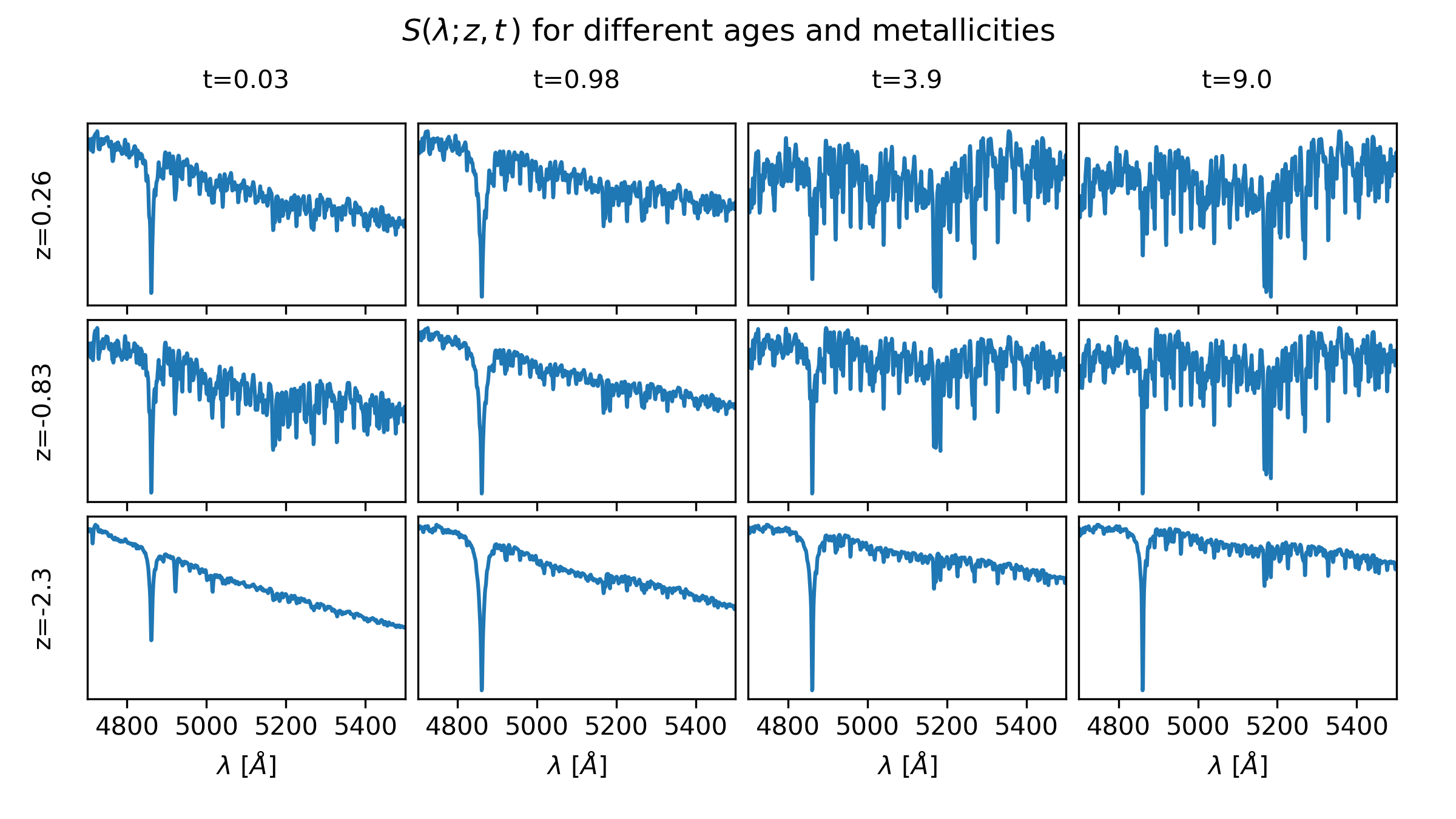}
    \caption{SSP templates $S(\lambda;z,t)$. We show the MILES SSP models as a function of wavelength at four ages $t$ labelled with units Gyr and three metallicities $z$ with units of [M/H] i.e. the logarithmic ratio of metal to hydrogen content, where we use the astronomy nomenclature of \emph{metal} to mean any element heavier than helium. Note that the range of the y-axis is scaled independently for each panel.}
    \label{fig_something_s}
\end{figure}

The relationship between $f$, $S$, and $y$ can be expressed as 
    \begin{equation}\label{eq_main_general}  
        \int_{\zmin}^{\zmax} \int_{\vmin}^{\vmax} \int_{\tmin}^{\tmax}  \frac{1}{1 + v/c} S\kl{\frac{\lambda}{1 + v/c} ; z, t} f(\xv, v, z, t) \, dt \, dv \, dz 
        \, = \,y(\xv, \lambda) \,,
    \end{equation}
where $c$ denotes the speed of light in vacuum. This is similar to the equation used to model a single spectrum e.g. see \cite[Equation~11]{Cappellari17}. The integrals over $z$ and $t$ allow for galaxies to host a superposition of stellar populations. The two factors of $1+v/c$ perform Doppler shifting of light given a line-of-sight velocity $v$: the factor inside the function $S$ shifts from rest-wavelength to the observed wavelength, while the factor outside ensures energy conservation, compensating for the wavelength stretching. A difference with respect to the single spectrum equation is that the full joint-distribution $f(\xv, v, z, t)$ appears in equation \eqref{eq_main_general}. In contrast, the single spectrum model must make the assumption that $f$ can be factorised as
    \begin{equation}\label{eq_factor_joint}  
        f(\xv, v, z, t) = f_1(\xv) f_2(v | \xv) f_3(t, z| \xv) \,.
    \end{equation}
This factorisation asserts that locally, kinematics and populations are independent. This assumption is known to be invalid in the Milky Way, where significant correlations between stellar age and velocity dispersion are observed in the solar neighbourhood \cite{Sharma21}. Hence, in addition to detecting weak features, a benefit of modelling the full datacube rather than single spectra is that we are no longer forced to assume the factorisation in equation \eqref{eq_factor_joint}.

\section{Mathematical model}\label{sect_model}

Equation \eqref{eq_main_general} can be written in a simpler form by grouping the variables $v,z,t$ together into the vector $\theta := (v,z,t)$ and defining
    \begin{equation}\label{def_kernel}
        k(\theta,\lambda) := k(v,z,t,\lambda) := \frac{1}{1 + v/c} S\kl{\frac{\lambda}{1 + v/c} ; z, t} \,.
    \end{equation}
With this, \eqref{eq_main_general} can be written in the form
    \begin{equation}\label{eq_main_compact}
        \int_\Theta k(\theta,\lambda) f(\xv,\theta) \, d \theta \, = \, y(\xv,\lambda) \,,
    \end{equation}   
where the set $\Theta \subset \R^3$ is defined by
    \begin{equation}\label{def_Theta}
        \Theta:= \Kl{\theta = (v,z,t) \in \R^3 \, \vert \, v \in (\vmin,\vmax) \,, z \in (\zmin,\zmax), t \in (\tmin,\tmax) } \,.
    \end{equation}
Equation \eqref{eq_main_compact} closely resembles a Fredholm integral equation of the first kind in standard form, which suggests to consider its solution within the framework of linear (ill-posed) operator equations \cite{Engl_Hanke_Neubauer_1996}. For this, we start by making the following 

\begin{assumption}\label{ass_main}
Let $\vmin < \vmax$, $\zmin < \zmax$, $\tmin < \tmax$, and $\lmin < \lmax$ be real numbers, and let $\Omega \subset \R^2$ be bounded and Jordan measurable. Furthermore, let $\Theta \subset \R^3$ be defined as in \eqref{def_Theta} and let $\Lambda := (\lmin,\lmax) \subset \R$. In addition, let the function $S = S(\lambda;z,t)$ be such that for $k$ defined as in \eqref{def_kernel} there holds $k\in \LtTL$.
\end{assumption}

The assumption that $k \in \LtTL$ is physically meaningful, as can also be seen from the spectra $S(\lambda;z,t)$ depicted in Figure~\ref{fig_something_s}. Using this, we now make the following
\begin{definition}
Let Assumption~\ref{ass_main} hold and let $0 \leq s \in \R$. Then we define
    \begin{equation}\label{def_operator_K}
    \begin{split}
        K \, : \, \HsOT \to \LtOL  \,,
        \qquad
        f \mapsto (K f)(\xv, \lambda) := \int_\Theta  k(\theta, \lambda) f(\xv, \theta) \,d\theta \,,
    \end{split}
    \end{equation}
where $\HsOT$ denotes the standard Sobolev space of order $s$ as defined e.g.~\cite{Adams_Fournier_2003}.
\end{definition}

With this definition, \eqref{eq_main_compact} can be rewritten as the linear operator equation
    \begin{equation}\label{Kf=y}
        K f = y \,.
    \end{equation}
Note that by choosing the Sobolev space $\HsOT$ as the definition space of $K$, we implicitly made the physically motivated assumption that the true distribution function $f$ is smooth in all variables. Hence, we arrive at the following mathematical

\begin{problem}\label{prob_K}
Let Assumption~\ref{ass_main} hold and let $y \in \LtOL$. Furthermore, let $0\leq s\in \R$ and let $K : \HsOT \to \LtOL$ be defined as in \eqref{def_operator_K}. Then the \emph{galactic archaeology problem with continuous wavelength measurements} consists in finding a function $f \in \HsOT$ which satisfies the operator equation \eqref{Kf=y}. 
\end{problem}

A number of basic properties of the integral operator $K$ are collected in
\begin{lemma}\label{lemma_properties_K}
Let Assumption~\ref{ass_main} hold and let $0\leq s\in \R$. Then the linear operator $K : \HsOT \to \LtOL$ defined in \eqref{def_operator_K} is bounded. Furthermore, the adjoint operator $K^*$ of $K$ is given by
    \begin{equation}\label{eq_K_adj}
        (K^*y)(\xv,\theta) =  E_s^*\kl{ \int_\Lambda  k(\theta, \lambda)  y(\xv,\lambda) \,d\lambda } \, ,
    \end{equation}
where $E_s : \HsOT \to \LtOT$ denotes the Sobolev embedding operator.
\end{lemma}
\begin{proof}
For any $f \in \HsOT$, using the Cauchy-Schwarz inequality we obtain
    \begin{equation*}
    \begin{split}
        &\norm{Kf}_\LtOL^2 = \int_{\Omega}\int_\Lambda \abs{\int_\Theta f(\xv, \theta) k(\theta, \lambda)d \theta }^2\, d\lambda  \, d\xv 
        \\
        & \qquad
        \leq  \int_{\Omega}\int_\Lambda \int_\Theta \abs{f(\xv, \theta)}^2  \,d\theta \int_\Theta \abs{ k(\theta, \lambda) }^2 \,d\theta  \, d\lambda  \, d\xv 
        =
        \norm{k}_\LtTL^2 \norm{f}_\LtOT^2 \,.
    \end{split}
    \end{equation*}
Together with the Sobolev embedding inequality \cite{Adams_Fournier_2003}
    \begin{equation}\label{Sobolev_embedding}
        \norm{f}_\LtOT \leq C_s \norm{f}_\HsOT \,, \qquad \forall \, s \in \R \,,
    \end{equation}
this yields boundedness of $K$. Next, for $f \in \HsOT$ and $y \in \LtOL$ consider  
    \begin{equation*}
    \begin{split}
        & \spr{Kf,y}_{\LtOL} = \int_\Omega \int_\Lambda \kl{\int_\Theta  k(\theta, \lambda) f(\xv, \theta) \,d \theta} y(\xv,\lambda) \, d\lambda \,d\xv
        \\
        & \quad
        = \int_\Omega \int_\Theta  \kl{    \int_\Lambda  k(\theta, \lambda)  y(\xv,\lambda) \, d\lambda} f(\xv, \theta)  \, d\theta \,d\xv 
        \\
        & \quad
        = \spr{ \int_\Lambda  k(\cdot_\theta, \lambda)  y(\cdot_{\xv},\lambda) \, d\lambda  \,, E_s f}_{\LtOT}
         = \spr{ E_s^* \kl{ \int_\Lambda  k(\cdot_\theta, \lambda)  y(\cdot_{\xv},\lambda) \, d\lambda },   f}_{\HsOT} \, , 
    \end{split}
    \end{equation*}
which establishes that the adjoint $K^*$ of $K$ is given as in \eqref{eq_K_adj}.
\end{proof}

In practice, measurement data $y(\xv,\lambda)$ are only available for finitely many wavelengths $\lambda_r$, $r \in \Kl{1\,,\dots\,,R}$. Hence, instead of via \eqref{eq_main_compact} one has to estimate the distribution $f$ from given data $y(\xv,\lambda_r)$ via the system of integral equations
    \begin{equation}\label{main_compact_system}
        \int_\Theta k(\theta,\lambda_r) f(\xv,\theta) \, d \theta \, = \, y(\xv,\lambda_r) \,, \qquad \forall \, r \in \Kl{1\,,\dots\,,R}\,.
    \end{equation}  
This in turn leads us to making the following
\begin{definition}
Let Assumption~\ref{ass_main} hold and let $R \in \N$. Furthermore, let $\lambda_r \in \Lambda$ for all $r \in \Kl{1\,,\dots\,,R}$. Then for any $0 \leq s \in \R$ we define the integral operators
    \begin{equation}\label{def_operator_K_r}
    \begin{split}
        K_r \, : \, \HsOT \to \LtO  \,,
        \qquad
        f \mapsto (K_r f)(\xv) := \int_\Theta  k(\theta, \lambda_r) f(\xv, \theta) \, d\theta \,,
    \end{split}
    \end{equation}
as well as the composite operator
    \begin{equation}\label{def_operator_Kv}
    \begin{split}
        \Kv \, : \, \HsOT \to \LtO^R\,,
        \qquad
        f \mapsto  \Kv f := \kl{K_1 f \,, \dots \,, K_R f } \,.
    \end{split}
    \end{equation}
\end{definition}

It follows with this definition that given a set of measurement data of the form
    \begin{equation}
        \yv = \kl{y_r}_{r=1}^R = \kl{y(x,\lambda_r) }_{r=1}^R \,,
    \end{equation}
the system of integral equations \eqref{main_compact_system} can be written as the linear operator equation    
    \begin{equation}\label{Kv_f=y}
        \Kv f = \yv \,.
    \end{equation}
This allows us to formulate the following mathematical

\begin{problem}\label{prob_Kv}
Let Assumption~\ref{ass_main} hold, let $R \in \N$ and let $\yv = (y_r)_{r=1}^R \in \LtO^R$. Furthermore, let $\lambda_r \in \Lambda$ for all $r \in \Kl{1\,,\dots\,,R}$, let $0\leq s\in \R$, and let $\Kv : \HsOT \to \LtO^R$ be defined as in \eqref{def_operator_Kv}. Then the \emph{galactic archaeology problem with discrete wavelength measurements} consists in finding a function $f \in \HsOT$ which satisfies the operator equation \eqref{Kv_f=y}. 
\end{problem}

A number of basic properties of the operators $\Kv$ and $K_r$ are collected in

\begin{lemma}\label{lemma_properties_A_Ar}
Let Assumption~\ref{ass_main} hold and let $R \in \N$. Furthermore, let $\lambda_r \in \Lambda$ for all $r \in \Kl{1\,,\dots\,,R}$ and let $0\leq s\in \R$. Then the linear operators $K_r: \HsOT \to \LtO$ and $\Kv: \HsOT \to \LtO^M$ defined as in \eqref{def_operator_K_r} and \eqref{def_operator_Kv}, respectively, are bounded. Furthermore, for any $\yv = (y_r)_{r=1}^R \in \LtO^R$ the adjoint operators $K_r^*$ of $K_r$ are given by 
    \begin{equation}\label{eq_Kr_adj}
        (K_r^* y_r)(\xv,\theta) =   E_s^*\kl{ k(\theta, \lambda_r)  y_r(\xv) }  \,,
    \end{equation}
and the adjoint $\Kv^*$ of the operator $\Kv$ is given by
    \begin{equation}\label{eq_Kv_adj}
        (\Kv^*\yv)(\xv,\theta) =  \sum_{r=1}^R (K_r ^* \, y_r)(\xv,\theta) 
        =  E_s^* \kl{ \sum_{r=1}^R k(\theta, \lambda_r)  y_r(\xv) }\,,
    \end{equation}
where as before $E_s : \HsOT \to \LtOT$ denotes the Sobolev embedding operator.
\end{lemma}
\begin{proof}
Let $f \in \HsOT$ be arbitrary but fixed and consider
    \begin{equation*}
    \begin{split}
        &\norm{K_r f}_\LtO^2 = \int_\Omega \abs{\int_\Theta  k(\theta, \lambda_r) f(\xv, \theta) \, d\theta}^2 \, d\xv 
        \\
        & \qquad
        \leq \int_\Omega \int_\Theta \abs{ k(\theta,\lambda_r)}^2\,d\theta  \int_\Theta \abs{f(\xv,\theta)}^2 \, d\theta\, d\xv 
        =
        \norm{k(\cdot,\lambda_r)}_\LtO^2 \norm{f}_\LtOT^2 \,.
    \end{split}     
    \end{equation*}
Together with \eqref{Sobolev_embedding} this yields the boundedness of the $K_r$ and thus, by its definition \eqref{def_operator_Kv} also the boundedness of $\Kv$. Next, for $f \in \HsOT$ and $y_r \in \LtO$ it follows with Fubini's theorem that
    \begin{equation*}
    \begin{split}
        & \spr{K_r f, y_r}_{\LtO} = \int_\Omega  \kl{\int_\Theta  k(\theta, \lambda_r) f(\xv, \theta) \,d \theta} y_r(\xv) \,d\xv
        \\
        & \qquad
        = \int_\Omega \int_\Theta   k(\theta, \lambda_r)  y_r(\xv) \, f(\xv, \theta)  \, d\theta \,d\xv
        \\
        & \qquad
        = \spr{   k(\cdot_\theta, \lambda_r)  y_r(\cdot_{\xv})  \,, E_s f}_{\LtOT}
        = \spr{ E_s^*\kl{  k(\cdot_\theta, \lambda_r)  y_r(\cdot_{\xv}) } \,, f}_{\HsOT} \,,
    \end{split}
    \end{equation*}
which establishes that the adjoint $K_r^*$ of $K_r$ is given as in \eqref{eq_Kr_adj}. Furthermore, by the definition \eqref{def_operator_Kv} of $\Kv$ it follows that for each $\yv= \kl{y_r}_{r=1}^R \in \LtO^R$ there holds
    \begin{equation*}
        \spr{\Kv f,\yv }_{\LtO^R} = \sum\limits_{r=1}^R \spr{K_r f , y_r}_\LtO 
        = \spr{f, \sum\limits_{r=1}^R K_r^* y_r, }_\HsOT \,,
    \end{equation*}
which shows that the adjoint $\Kv^*$ is given as in \eqref{eq_Kv_adj}, which concludes the proof. 
\end{proof}

Note that for the function $f$ to be physically meaningful, there has to hold
    \begin{equation}\label{constraint_nonnegative}
        f(\xv,v,z,t) = f(\xv,\theta) \geq 0 \,, 
        \qquad
        \forall \, \xv \in \Omega \,, \, \theta = (v,z,t) \in \Theta  \,.
    \end{equation}
This has to be kept in mind when solving the galactic archaeology problem and shall be addressed when considering solution approaches in the subsequent section. 
    
\section{Reconstruction approach - the PNKR method}\label{sect_reconstruction}

In this section, we consider different reconstruction approaches for the galactic archaeology Problem~\ref{prob_Kv}, which amounts to solving the linear operator equation \eqref{Kv_f=y}, i.e.,
    \begin{equation*}
        \Kv f = \yv \,.
    \end{equation*}
In practice, instead of exact data $\yv = (y_r)_{r=1}^R$ one typically only has access to noisy measurement data $\yvd = (\yd_r)_{r=1}^R$, which can be assumed to satisfy the estimate
    \begin{equation*}
        \norm{\yd_r - y_r} \leq \delta_r 
    \end{equation*}   
for some given noise levels $\delta_r$. Defining the total noise level
    \begin{equation*}
        \delta^2 := \sum\limits_{r=1}^R\delta_r^2  \,,
    \end{equation*}
our given data $\yvd$ are thus assumed to satisfy the the estimate
     \begin{equation*}
        \norm{\yvd - \yv} \leq \delta \,.
    \end{equation*}
In general, solving \eqref{Kv_f=y} is an ill-posed problem, with the degree of ill-posedness depending on the value of $s$ in the definition space $\HsOT$ of $\Kv$. Thus, regularization is necessary in order to obtain stable approximations of solutions to \eqref{Kv_f=y}, in particular in the presence of noisy data $\yvd$; see e.g.~\cite{Engl_Hanke_Neubauer_1996,Louis_1989}. Hence, in this section we first review some necessary background on regularization methods, in particular on Kaczmarz methods. These have the benefit of leveraging the system structure of the problem, and form the basis of the PNKR method for its solution which we then introduce below.

\subsection{Variational and iterative regularization methods}

One of the most popular regularization method is Tikhonov regularization \cite{Engl_Hanke_Neubauer_1996,Tikhonov_1963}, which determines an approximation $f_\alpha^\delta$ to the minimum-norm solution $f^\dagger$ by minimizing
    \begin{equation*}
        \mathcal{T}_\alpha^\delta(f) := \norm{\Kv f - \yvd}^2 + \alpha \norm{ f }^2 \,.
    \end{equation*}
The minimizer of this functional can be stated explicitly as
    \begin{equation}\label{eq_fad}
        f_\alpha^\delta = \kl{\Kv^*\Kv + \alpha I}^{-1} \Kv^* \yvd \,, 
    \end{equation}
where $I$ denotes the identity operator. Tikhonov regularization has been used successfully in many applications, in particular since it can be adapted to include many different assumptions on the sought for solution such as regularity or sparsity with respect to a given basis. Alternatively, one can directly use iterative regularization methods for solving \eqref{Kv_f=y}. Popular choices include e.g.\ Landweber iteration, given by
    \begin{equation}\label{Landweber}
        \fkpd = \fkd + \omega \Kv^*(\yvd - \Kv \fkd ) \,,
    \end{equation}
where $\omega$ is a stepsize parameter, or conjugate-gradient for the normal equations (CGNE), which is the method of conjugate gradients (CG) applied to
    \begin{equation*}
         \Kv^*\Kv f = \Kv^* \yvd \,.
    \end{equation*}
In order to be regularization methods, iterative regularization methods usually need to be paired with a suitable stopping rule to terminate the iteration. A very popular choice is the discrepancy principle, which determines the stopping index $k^*$ as
    \begin{equation}\label{discrepancy_principle}
        k^* := \inf\Kl{ k \in \N \, \vert \, \norm{\Kv \fkd - \yvd} \leq \tau \delta  } \,,
    \end{equation}
for some parameter $\tau > 1$. Among iterative methods, CGNE often requires the smallest amount of iterations if the discrepancy principle is used for stopping the iteration \cite{Engl_Hanke_Neubauer_1996}. However, it is not always the best choice in practice: For example, incorporating additional assumptions on the solution such as sparsity constraints is not straightforward, as the method itself is not very flexible. On the other hand, this is particularly easy to do with Landweber iteration and its many different variants, leading e.g. in the case of sparsity constraints to such methods as ISTA. Furthermore, even though in its basic form Landweber iteration is known to be quite slow, many successful attempts have been made to speed it up considerably \cite{Hanke_1991, Neubauer_2000, Ramlau_1999, Nesterov_1983, Neubauer_2017, Scherzer_1996, Attouch_Peypouquet_2016, Hubmer_Ramlau_2017, Hubmer_Ramlau_2018, Kindermann_2021, Jin_2016_01}.

However, both for CGNE and most of the variants of Landweber iteration, the operators $\Kv$ and $\Kv^*$ need to be applied at each iteration step. While typically of no concern in small or medium scale problems, for large scale problems such as our galactic archaeology problem, the numerical effort required to do so can often become problematic. Although in some applications a model reduction or a lower order approximation of the operators $\Kv$ and $\Kv^*$ can be employed as an effective remedy, here we propose to use the underlying system structure of \eqref{Kv_f=y} to our advantage. In particular, we propose to use a specific modification of the Landweber-Kaczmarz approach outlined below.

The classical Landweber-Kaczmarz method for solving \eqref{Kv_f=y} is defined by 
    \begin{equation}\label{LW_KM_classic}
        \fkpd = \fkd + \omega K_\kmod^*(\yd_\kmod - K_\kmod \fkd ) \,,
    \end{equation}
where $\omega$ is again a stepsize-parameter and
    \begin{equation}\label{def_kmod}
        \kmod := (k \operatorname{mod} R) + 1 \quad \in \Kl{1\,,\dots\,,R} \,.
    \end{equation}
Note that in each step of this method only the operators $K_\kmod$ and $K_\kmod^*$, but not the full operators $\Kv$ and $\Kv^*$ need to be applied, which makes each iteration step computationally more efficient than, e.g., classical Landweber iteration \eqref{Kv_f=y}. Concerning the choice of a suitable stopping rule, note that while it is possible to use the discrepancy principle \eqref{discrepancy_principle}, here it also makes sense to check the component-wise discrepancy
    \begin{equation*}
        \norm{K_\kmod \fkd - \yd_\kmod}  \,,
    \end{equation*}
and only update the current iterate if this quantity is larger than $\tau \delta_\kmod$. This can be effected by defining the iteration dependent stepsize 
    \begin{equation}\label{def_okd}
        \okd 
        :=
        \begin{cases}
            \omega  \,,  & \norm{K_\kmod \fkd - \yd_\kmod} > \tau \delta_\kmod \,,
            \\
            0 \,, & \text{else} \,,
        \end{cases}
    \end{equation}
and instead of \eqref{LW_KM_classic} using the adapted iterative procedure
    \begin{equation}\label{LW_KM_cDP}
        \fkpd = \fkd + \okd K_\kmod^*(\yd_\kmod - K_\kmod \fkd ) \,.
    \end{equation}
The convergence properties of this adapted Landweber-Kaczmarz method have been analysed previously both for linear and for nonlinear operators \cite{Haltmeier_Leitao_Scherzer_2007,DeCezaro_Haltmeier_Leitao_Scherzer_2008,Kindermann_Leitao_2014,Engl_Hanke_Neubauer_1996,Kaltenbacher_Neubauer_Scherzer_2008}. Note that the constant stepsize $\omega$ in \eqref{def_okd} can also be replaced by an iteration dependent stepsize $\akd$ such as the steepest descent or the minimal error stepsize \cite{Kaltenbacher_Neubauer_Scherzer_2008}.

We propose to adapt the Landweber-Kaczmarz method \eqref{LW_KM_cDP} for solving the extragalactic archaeology Problem~\ref{prob_Kv} by including two modifications of the method concerning efficiency and physical prior information as outlined below. Together with a specific discretization this will then yield the PNKR method. In order to define this discretization, we first need to consider a general discretization of the Landweber-Kaczmarz method.

\subsection{The discrete Landweber-Kaczmarz method}

There exist a number of ways to derive discrete versions of (iterative) regularization methods such as the Landweber-Kaczmarz method; see e.g.\ \cite{Engl_Hanke_Neubauer_1996}. These commonly involve projections onto finite dimensional subspaces of the definition and/or image space of the underlying operator. Here, we consider the finite dimensional subspaces
    \begin{equation}\label{def_XN_YM}
    \begin{split}
        \XM &:= \operatorname{span}\Kl{\phi_1(\xv,\theta) \,, \dots \,, \phi_M(\xv,\theta)} \subset \HsOT \,,
        \\
        \YN &:= \operatorname{span}\Kl{\psi_1(\xv) \,, \dots \,, \psi_N(\xv)}    \subset \LtO \,,
    \end{split}
    \end{equation}
where the functions $\phi_m(\xv,\theta)$ and $\psi_m(\xv)$ denote linearly independent basis functions to be specified below. Next, denote by $P$ and $Q$ the orthogonal projectors onto $\XM$ and $\YN$, respectively, and define the operators $\Kt_r := Q K_r P$. With this, we can introduce the following discrete version of the Landweber-Kaczmarz method \eqref{LW_KM_cDP}:
    \begin{equation}\label{LW_KM_cDP_discrete}
        \fkpdt = \fkdt + \okdt \Kt_\kmod^*(Q \yd_\kmod - \Kt_\kmod \fkdt )\,,
    \end{equation}
where the discrete version $\okdt$ of the stepsize $\okd$ is given by
    \begin{equation}\label{stepsize_general_discrete}
        \okdt 
        :=
        \begin{cases}
            \omega  \,,  & \norm{\Kt_\kmod \fkd - Q \yd_\kmod} > \tau \delta_\kmod \,,
            \\
            0 \,, & \text{else} \,.
        \end{cases}
    \end{equation}    
This iteration \eqref{LW_KM_cDP_discrete} can be written in a matrix-vector form, for which we first need 

\begin{definition}
Let $s \in \R$, let $\Theta \subset \R^3$ be defined as in \eqref{def_Theta}, and let $\Omega \subset \R^2$ be bounded and Jordan measurable. Furthermore, let $\Kl{\phi_m}_{m=1}^M \subset \HsOT$ and $\Kl{\psi_n}_{n=1}^N \subset \LtO$ denote two sets of linearly independent functions and let the finite dimensional spaces $\XM$ and $\YN$ be defined as in \eqref{def_XN_YM}. Then we define the matrices
    \begin{equation}\label{def_matrices_M_N}
        \Mv := \kl{ \spr{\phi_i,\phi_j}_\XM }_{i,j=1}^M  \,,
        \qquad
        \text{and}
        \qquad
        \Nv := \kl{ \spr{\psi_i,\psi_j}_\YN }_{i,j=1}^N \,.
    \end{equation}
Furthermore, for all $r = 1, \dots, R$, for $\yd_r \in Y$ and with $K_r$ as in \eqref{def_operator_K_r} we define  
    \begin{equation}\label{def_Hr_wrd}
        \Hv_r := \kl{ \spr{ \psi_j, K_r \phi_i  }_\YN }_{j,i=1}^{N,M} \,,
        \qquad
        \text{and}
        \qquad
        \wv_r^\delta := \kl{ \spr{\yd_r,\psi_j}_\YN }_{j=1}^N \,.
    \end{equation}
\end{definition}
    
The matrices $\Mv$ and $\Nv$ can be used to characterize the projections onto the subspaces $X_M$ and $Y_N$, respectively. For example, for any $\Tilde{w} \in Y$ there holds
    \begin{equation}\label{connection_vector_function}
        Q  \Tilde{w} = \sum\limits_{j=1}^N \Tilde{w}_j    \psi_j 
        \qquad
        \Longleftrightarrow
        \qquad
        \Nv \kl{ \tilde{w}_j  }_{j=1}^N = \kl{ \spr{\tilde{w} , \psi_j}}_{j=1}^N \,,
    \end{equation}
and similarly for the projector $P$ onto $X_M$. Using this, we can now follow the standard approach for discretizing iterative regularization methods, see e.g.~\cite{Engl_Hanke_Neubauer_1996}, and derive

\begin{proposition}
Let $s \in \R$, let $\Theta \subset \R^3$ be defined as in \eqref{def_Theta}, and $\Omega \subset \R^2$ be bounded and Jordan measurable. Furthermore, let $\Kl{\phi_m}_{m=1}^M \subset \HsOT$ and $\Kl{\psi_n}_{n=1}^N \subset \LtO$ denote two sets of linearly independent functions and let the finite dimensional spaces $\XM$ and $\YN$ be defined as in \eqref{def_XN_YM}. Furthermore, for all $r=1,\dots,R$ let $K_r$ be as in \eqref{def_operator_K_r}, $\yd_r \in Y$ and let $\Mv$, $\Nv$, $\Hv_r$, and $\wv_r^\delta$ be defined as in \eqref{def_matrices_M_N} and \eqref{def_Hr_wrd}, respectively. Then for the iterates $\fkdt$ defined in \eqref{LW_KM_cDP_discrete} there holds
    \begin{equation}\label{eq_expansion}
        \fkdt = \sum\limits_{i=1}^M u_{i,k}^\delta \phi_i \,,
    \end{equation}
where the coefficients $\uv_k^\delta = \kl{ u_{i,k}^\delta }_{i=1}^M$ are given by the  iteration
    \begin{equation}\label{LW_KM_stepsize_discrete}
        \uvkpd = \uvkd + \okdt \Mv^{-1} \Hv_\kmod^T \Nv^{-1} \kl{ \wv_\kmod^\delta -  \Hv_\kmod  \uvkd } \,,
    \end{equation}
which is also called the \emph{discrete Landweber-Kaczmarz method}.
\end{proposition}
\begin{proof}
Since the functions $\phi_m$ form a basis of $\XM$, the iteration \eqref{LW_KM_cDP_discrete} is equivalent to
    \begin{equation}\label{eq_LW_tested}
    \begin{split}
        \spr{\fkpdt, \phi_m}_\XM = \spr{\fkdt, \phi_m}_\XM + \okdt \spr{\Kt_\kmod^*(Q \yd_\kmod - \Kt_\kmod \fkdt ), \phi_m}_\XM\,,
        \\ \qquad
        \forall \, m \in \Kl{ 1 \,, \dots \,, M} \,.
    \end{split}
    \end{equation}
Now since by definition $\fkdt \in X_M$ for any $k\in\N$, it can be decomposed as in \eqref{eq_expansion} for some coefficients $\uv_k^\delta = \kl{ u_{i,k}^\delta }_{i=1}^M$. Inserting \eqref{eq_expansion} into \eqref{eq_LW_tested} we thus obtain
    \begin{equation*}
    \begin{split}
        &\spr{\sum\limits_{i=1}^M \uikpd \phi_i, \phi_m}_\XM 
        = 
        \spr{\sum\limits_{i=1}^M \uikd \phi_i, \phi_m}_\XM 
        \\
        & \qquad 
        + 
        \okdt \spr{\Kt_\kmod^*\kl{Q \yd_\kmod
        - 
        \Kt_\kmod \kl{ \sum\limits_{i=1}^M \uikd \phi_i}  }, \phi_m}_\XM\,,
        \qquad
        \forall \, m \in \Kl{ 1 \,, \dots \,, M} \,.
    \end{split}
    \end{equation*}
Using the linearity of the inner product and the fact that $\Kt_r = Q K_r P$ we find that
    \begin{equation}\label{eq_LK_tested_basis}
    \begin{split}
        &\sum\limits_{i=1}^M \uikpd \spr{ \phi_i, \phi_m}_\XM
        = 
        \sum\limits_{i=1}^M \uikd  \spr{\phi_i, \phi_m}_\XM 
        \\
        & \qquad
        + 
        \okdt \spr{ K_\kmod^* Q\kl{ \yd_\kmod
        -  \sum\limits_{i=1}^M \uikd K_\kmod \phi_i}, \phi_m}_\XM \,,
        \qquad
        \forall \, m \in \Kl{ 1 \,, \dots \,, M} \,.
    \end{split}
    \end{equation}
Next, we use \eqref{connection_vector_function} for the choice $\tilde{w} := \yd_\kmod$ and obtain that
    \begin{equation}\label{eq_QN_ykd}
    \begin{split}
        & Q \yd_\kmod 
        =
        \sum\limits_{j=1}^N  \kl{ \Nv^{-1} \kl{ \spr{\yd_\kmod,\psi_n}  }_{n=1}^N   }_j  \psi_j 
        = 
        \sum\limits_{j=1}^N  \kl{ \Nv^{-1} \wv_\kmod^\delta }_j  \psi_j \,.
    \end{split}
    \end{equation}
Similarly, we find that
    \begin{equation*}
    \begin{split}
        Q  \kl{\sum\limits_{i=1}^M \uikd K_\kmod \phi_i }
        =
        \sum\limits_{j=1}^N  \kl{ \Nv^{-1} \kl{ \sum\limits_{i=1}^M \uikd  \spr{ K_\kmod \phi_i   ,\psi_n}  }_{n=1}^N   }_j  \psi_j 
        =
        \sum\limits_{j=1}^N  \kl{ \Nv^{-1}  \Hv_\kmod  \uikd  }_j  \psi_j \,,
    \end{split}
    \end{equation*}
such that together with \eqref{eq_QN_ykd} we get
    \begin{equation*}
        Q \kl{ \yd_\kmod - \sum\limits_{i=1}^M \uikd K_\kmod \phi_i }
        = \sum\limits_{j=1}^N  \kl{ \Nv^{-1} \kl{ \wv_\kmod^\delta -  \Hv_\kmod  \uikd  } }_j \psi_j \,.
    \end{equation*}
Inserting this expression into \eqref{eq_LK_tested_basis} we obtain
    \begin{equation*}
    \begin{split}
        &\sum\limits_{i=1}^M \uikpd \spr{ \phi_i, \phi_m}_\XM 
        = 
        \sum\limits_{i=1}^M \uikd  \spr{\phi_i, \phi_m}_\XM 
        \\
        & \qquad
        + 
        \okdt \sum\limits_{j=1}^N  \kl{ \Nv^{-1} \kl{ \wv_\kmod^\delta -  \Hv_\kmod  \uikd  } }_j \spr{   K_\kmod^*  \psi_j, \phi_m}_\XM
        \qquad
        \forall \, m \in \Kl{ 1 \,, \dots \,, M} \,,
    \end{split}
    \end{equation*}
and thus together with \eqref{def_matrices_M_N} find that
    \begin{equation*}
        \Mv \uvkpd = \Mv \uvkd + \okdt \Hv_\kmod^T \Nv^{-1} \kl{ \wv_\kmod^\delta -  \Hv_\kmod  \uvkd} \,.
    \end{equation*}
Finally, after applying the inverse of $\Mv$ we obtain \eqref{LW_KM_stepsize_discrete}, which yields the assertion.
\end{proof}

Concerning the implementation of $\okdt$ defined in \eqref{stepsize_general_discrete}, note that there holds
    \begin{equation*}
        \norm{\Kt_\kmod \fkd - Q \yd_\kmod}^2 
        =
        \kl{ \Nv^{-1} \kl{ \wv_\kmod^\delta -  \Hv_\kmod  \uvkd }}^T \Nv \kl{ \Nv^{-1} \kl{ \wv_\kmod^\delta -  \Hv_\kmod  \uvkd } } \,,
    \end{equation*}
completing the matrix-vector form of the discrete Landweber-Kaczmarz method \eqref{LW_KM_stepsize_discrete}.



\subsection{Modifications of the Landweber-Kaczmarz method}\label{modifications}

The Landweber-Kaczmarz method in both its discrete and continuous forms allows for a number of useful modifications in terms of efficiency and incorporation of physical prior information. In this section, we discuss four modifications which we then combine together with a specific discretization into the PNKR method introduced below.

The first modification concerns the inclusion of the non-negativity constraint \eqref{constraint_nonnegative}, which has not yet entered into the (discrete) Landweber-Kaczmarz method. A common way to incorporate it is via the definition of an (orthogonal) projection/cutoff operator
    \begin{equation*}
    \begin{split}
        \mathcal{T} \, : \, \HsOT &\to \Kl{ f \in \HsOT \, \vert \, f \geq 0 } \,, 
        \\
        (\mathcal{T}f)(x) &:= 
        \begin{cases}
           f(x) \,, & f(x) \geq 0 \,,
           \\
           0 \,, & \text{else} \,,
        \end{cases}
    \end{split}
    \end{equation*}
which is used to adapt iterative methods by projecting the update after each step \cite{Engl_Hanke_Neubauer_1996}. For example, for the classical Landweber-Kaczmarz method \eqref{LW_KM_classic} this adaptation reads
    \begin{equation}\label{LW_KM_classic_nonnegative}
        \fkpd = \mathcal{T}\kl{\fkd + \omega K_\kmod^*(\yd_\kmod - K_\kmod \fkd ) }\,,
    \end{equation}
and analogously for most other iterative method. For the discrete Landweber-Kaczmarz method \eqref{LW_KM_stepsize_discrete} this approach can be effectively realized as follows: First, we define
    \begin{equation*}
         h \, : \R \to \R_0^+\,, 
         \qquad
         x \mapsto 
         \begin{cases}
            x \,, &  x \geq 0 \,,
            \\
            0 \,, & \text{else} \,,
         \end{cases}
    \end{equation*}
and use this function to define the thresholding operator
    \begin{equation}\label{def_op_T}
        T \, : \, \R^M \to \R^M \,, 
        \qquad
        \uv = \Kl{u_i}_{i=1}^M \mapsto T \uv := \Kl{h(u_i) }_{i=1}^M \,.
    \end{equation}
Using this, the discrete Landweber-Kaczmarz method \eqref{LW_KM_cDP_discrete} can be adapted to
    \begin{equation*}\label{LW_KM_cDP_discrete_T}
        \uvkpd = T\kl{\uvkd + \okdt \Mv^{-1} \Hv_\kmod^T \Nv^{-1} \kl{ \wv_\kmod^\delta -  \Hv_\kmod  \uvkd } } \,,
    \end{equation*}
which results in the obtained reconstruction being non-negative.

The second modification concerns the numerical efficiency of the method. While the Landweber-Kaczmarz method already leverages the system structure of the problem, it may still require too many iterations to be feasible in practice. For our problem, this is particularly true when the total number $R$ of considered wavelengths $\lambda_r$ is large. Hence, we propose to apply Nesterov acceleration \cite{Nesterov_1983} to method \eqref{LW_KM_cDP_discrete_T}, which yields
    \begin{equation}\label{nesterov}
    \begin{split}
        \zv_k^\delta &= \uvkd + \frac{k_R-1}{k_R+2}(\uvkd-\uv_{k-1}^\delta) \,,
        \qquad
        k_R = \operatorname{ceil}(k/R)\,,
        \\ 
        \uvkpd &=  T\kl{ \zv_k^\delta  + \okdt \Mv^{-1} \Hv_\kmod^T \Nv^{-1} \kl{ \wv_\kmod^\delta -  \Hv_\kmod  \zv_k^\delta  } }\,.
    \end{split}
    \end{equation} 
Theoretical and numerical properties of Nesterov(-type) acceleration applied to (projected) Landweber and Landweber-Kaczmarz methods for solving inverse problems, in particular convergence and convergence rates results, were studied for example in \cite{Hubmer_Ramlau_2017,Hubmer_Ramlau_2018,Jin_2016_01,Kindermann_2021,Nesterov_1983,Attouch_Peypouquet_2016,Neubauer_2017,Beck_Teboulle_2009,Long_Han_Tong_2019}. We also find that it significantly reduces the required number of iterations in practice for our extragalactic archaeology problem.

The third modification concerns the choice of inner product on the Sobolev space $\HsOT$, which canonically is defined using multiindices $\alpha = (\alpha_1\,, \dots \,, \alpha_5)$ by
    \begin{equation*}
        \spr{f,g}_\HsOT = \sum\limits_{0\leq \abs{\alpha} \leq s} \spr{ \partial^\alpha f, \partial^\alpha g}_{\LtOT} \,.
    \end{equation*}
However, for each sequence $\bv = \Kl{\beta_\alpha}_{0 \leq \abs{\alpha} \leq s}$ of constants $\beta_\alpha > 0$ the inner product
    \begin{equation}\label{beta_vector}
        \spr{f,g}_{s,\bv} := \sum\limits_{0\leq \abs{\alpha} \leq s} \beta_\alpha \spr{ \partial^\alpha f, \partial^\alpha g}_{\LtOT} \,,
    \end{equation}
induces an equivalent norm on $\HsOT$. By tuning the parameters $\beta_\alpha$ one can thus place an emphasis on different derivatives in the inner product. This is useful if one wants to achieve a balance between derivatives of different orders or put an emphasis on smoothness in certain components. For example, in the case $s=1$ we will later use
    \begin{equation} \label{beta}
        \spr{f,g}_{\s,\bv} 
        := 
        \spr{f, g}_{\LtOT} 
        +
        \beta \sum\limits_{\abs{\alpha} = 1} \spr{\partial^\alpha f, \partial^\alpha g}_{\LtOT} \,,
    \end{equation}
and tune the constant $\beta$ to balance the $\LtOT$ component and the first-order derivatives. If one uses an inner product $\spr{\cdot,\cdot}_\beta$ then one also has to use it on the subspace $X_M \subset \HsOT$. This influences the matrix $\Mv$ defined in \eqref{def_matrices_M_N}, which now reads
    \begin{equation*}
        \Mv := \kl{ \spr{\phi_i,\phi_j}_{s,\bv} }_{i,j=1}^M  \,,
    \end{equation*}
and is the only point of influence on the discrete Landweber-Kaczmarz method. 
 
The final modification concerns the order of the operator equations $K_r f = y_r$. The Landweber-Kaczmarz method cycles through these equations in a loop induced by the index $\kmod$ defined in \eqref{def_kmod}. However, this is not strictly necessary, and it has been found that changing the order of equations in either a systematic or even a random fashion can lead to considerable improvements in terms of efficiency and reconstruction quality.

\subsection{The PNKR method}

In this section, we combine our previous considerations with a suitable choice of basis functions $\phi_m$ and $\psi_n$ to arrive at the PNKR method. For this, note first that the main computational effort required for applying \eqref{nesterov} stems from the following sources: 
\begin{enumerate}
    \item In each iteration, the matrices $\Hv_\kmod$ and $\Hv_\kmod^T$ need to be applied.
    \item In each iteration, two systems of equations with $\Mv$ and $\Nv$ need to be solved.
\end{enumerate}
In order to keep the computational costs of these steps as low as possible, the matrices $\Mv$, $\Nv$, and $\Hv_r$ should be as sparse and structured as possible. In particular for the matrices $\Mv$ and $\Nv$, this essentially depends on the choice of the functions $\phi_m$ and $\psi_n$.

Consider for example the typical case that the domain $\Omega$ is rectangular. Then a possible choice for the functions $\psi_n$ are the piecewise constant functions on a uniform rectangular subdivision of $\Omega$. This choice has the advantage that the matrix $\Nv$ turns into a scaled identity matrix, which is trivial to store, apply, and invert. Similarly, given a uniform rectangular subdivision of $\Omega \times \Theta$, the functions $\phi_m$ can e.g.\ be chosen as piecewise polynomial functions satisfying the smoothness requirement $\phi_m \in \HsOT$. This choice results in a highly structured matrix $\Mv$ that can be efficiently handled numerically, e.g.\ using methods from the theory of finite elements \cite{Braess_2007,Necas_2011,Jung_Langer_2012}.

Another possibility for choosing the functions $\psi_m$ and $\phi_n$ is suggested by the structure of the problem itself: Start by taking two linearly independent sets of functions $\Kl{\psi_n}_{n=1}^N \subset H^s(\Omega)$ and $\Kl{\varphi_l}_{l=1}^L \subset H^s(\Theta)$ and define the functions
    \begin{equation}\label{def_phi_m_product}
        \phi_m(\xv,\theta) := \psi_{n(m)}(\xv)\varphi_{l(m)}(\theta) \,,
        \qquad
        \forall \, m \in \Kl{1\,,\dots\,,M} \,, 
    \end{equation}
where $M:= N\cdot L$, $n(m) := \operatorname{floor}\kl{(m-1)/L} + 1$, and $l(m) := ((m-1) \mod L) + 1$. The resulting set of linearly independent functions $\Kl{\phi_m}_{m=1}^M \subset \HsOT$ can then be used for defining the finite dimensional subspace $\XM$ as in \eqref{def_XN_YM}. This has some beneficial consequences for the structure of the matrices $\Mv$, $\Nv$, and $\Hv_r$, as we see in

\begin{proposition}
Let $s \in \R$, $\Theta \subset \R^3$ be defined as in \eqref{def_Theta}, and $\Omega \subset \R^2$ be bounded and Jordan measurable. Furthermore, let $\Kl{\psi_n}_{n=1}^N \subset H^s(\Omega)$ and $\Kl{\varphi_l}_{l=1}^L \subset H^s(\Theta)$ denote two sets of linearly independent functions and let $\Kl{\phi_m}_{m=1}^M \subset \HsOT$ be defined as in \eqref{def_phi_m_product}. Moreover, let the finite dimensional spaces $\XM$ and $\YN$ be defined as in \eqref{def_XN_YM}, and $K_r$ be defined as in \eqref{def_operator_K_r} for all $r=1,\dots,R$. Then for the matrices $\Mv$, $\Nv$, and $\Hv_r$ defined as in \eqref{def_matrices_M_N} and \eqref{def_Hr_wrd} there holds
\begin{equation}\label{def_matrices_M_N_Hr_alternative}
    \begin{split}
        &\Mv = \kl{ \spr{\psi_{n(i)},\psi_{n(j)}}_{H^s(\Omega)}
         \spr{\varphi_{l(i)},\varphi_{l(j)}}_{H^s(\Theta)} }_{i,j=1}^M \,,
         \qquad
        \Nv = \kl{ \spr{\psi_i,\psi_j}_\LtO }_{i,j=1}^N \,,
        \\
        &\Hv_r = \kl{\spr{ \psi_j , \psi_{n(i)}}_\LtO
        \kl{\int_\Theta  k(\theta,\lambda_r)  \varphi_{l(i)}(\theta) \, d\theta }  }_{i,j=1}^{M,N} \,.
    \end{split}
\end{equation}
\end{proposition}
\begin{proof}
First of all, due to the definition \eqref{def_phi_m_product} of the functions $\phi_m$ there holds
    \begin{equation*}
    \begin{split}
         \spr{\phi_i,\phi_j}_\HsOT 
         &= 
         \spr{\psi_{n(i)}\varphi_{l(i)},\psi_{n(j)}\varphi_{l(j)}}_\HsOT
         \\
         &= 
         \spr{\psi_{n(i)},\psi_{n(j)}}_{H^s(\Omega)}
         \spr{\varphi_{l(i)},\varphi_{l(j)}}_{H^s(\Theta)} \,,
    \end{split}
    \end{equation*}
and thus together with definition \eqref{def_matrices_M_N} the specific form of $\Mv$ follows. Furthermore, concerning the matrices $\Hv_r$, it follows from \eqref{eq_Kr_adj} that
    \begin{equation*}
        \spr{ \psi_j, K_r \phi_i }_\LtO
        = 
        \spr{K_r^* \psi_j, \phi_i}_\HsOT
        \overset{\eqref{eq_Kr_adj}}{=}
        \spr{E_s^*\kl{k(\cdot_\theta,\lambda_r) \psi_j (\cdot_{\xv} )} ,  \phi_i   }_\HsOT
    \end{equation*}
and thus together with \eqref{def_phi_m_product},
    \begin{equation*}
    \begin{split}
       \spr{ \psi_j, K_r \phi_i }_\LtO
        &=
        \spr{\kl{k(\cdot_\theta,\lambda_r) \psi_j (\cdot_{\xv} )} ,  \phi_i   }_\LtOT
        =
        \int_\Omega \int_\Theta  k(\theta,\lambda_r) \psi_j (\xv) \phi_i(\omega,\theta) \, d\theta \, d\omega  
        \\
        &\overset{\eqref{def_phi_m_product}}{=}
        \int_\Omega \int_\Theta  k(\theta,\lambda_r) \psi_j (\xv)    
        \psi_{n(i)}(\xv) \varphi_{l(i)}(\theta) \, d\theta \, d\xv  \,,
    \end{split} 
    \end{equation*}
which can be rewritten into
    \begin{equation}\label{helper_2}
        \spr{ \psi_j, K_r \phi_i }_\LtO
        =
        \spr{ \psi_j , \psi_{n(i)}}_\LtO
        \kl{\int_\Theta  k(\theta,\lambda_r)  \varphi_{l(i)}(\theta) \, d\theta } \,.
    \end{equation}
Together with the definition \eqref{def_Hr_wrd} of the matrices $\Hv_r$ this yields the assertion.
\end{proof}

The specific form \eqref{def_matrices_M_N_Hr_alternative} of the matrices $\Mv$ and $\Hv_r$ induced by \eqref{def_phi_m_product} has some beneficial consequences concerning their practical implementation. For example, since 
    \begin{equation*}
    \begin{split}
         \spr{\phi_i,\phi_j}_\HsOT 
         &= 
         \spr{\psi_{n(i)}\varphi_{l(i)},\psi_{n(j)}\varphi_{l(j)}}_\HsOT
         \\
         &= 
         \spr{\psi_{n(i)},\psi_{n(j)}}_{H^s(\Omega)}
         \spr{\varphi_{l(i)},\varphi_{l(j)}}_{H^s(\Theta)} \,,
    \end{split}
    \end{equation*}
in the case $s=0$ the scalar products $\spr{\psi_{n(i)},\psi_{n(j)}}_{H^s(\Omega)} = \spr{\psi_{n(i)},\psi_{n(j)}}_\LtO$ need to be computed only once for both the matrices $\Mv$ and $\Nv$. Furthermore, from \eqref{def_matrices_M_N_Hr_alternative} we can see that for assembling both the matrix $\Nv$ and the matrices $\Hv_r$ the same scalar products $\spr{ \psi_j , \psi_{n(i)}}_\LtO$ need to be computed, which reduces the numerical effort.

With the specific discretization discussed above, we have all ingredients to define

\begin{PNKR}\label{method_PNKR}
Let $s \in \R$, let $\Theta \subset \R^3$ be defined as in \eqref{def_Theta}, and let $\Omega \subset \R^2$ be bounded and Jordan measurable. Furthermore, let $\Kl{\psi_n}_{n=1}^N \subset H^s(\Omega)$ and $\Kl{\varphi_l}_{l=1}^L \subset H^s(\Theta)$ denote two sets of linearly independent functions and let $\Kl{\phi_m}_{m=1}^M \subset \HsOT$ be defined as in \eqref{def_phi_m_product}. Moreover, let the finite dimensional spaces $\XM$ and $\YN$ be defined as in \eqref{def_XN_YM}, and for all $r=1,\dots,R$ let $K_r$ be defined as in \eqref{def_operator_K_r} and let $\yd_r \in Y$. Then the \emph{projected Nesterov-Kaczmarz reconstruction (PNKR) method} for solving the galactic archaeology Problem~\ref{prob_Kv} reads
    \begin{equation}\label{PNKR}
    \begin{aligned}
        \zv_k^\delta &= \uvkd + \frac{k_R-1}{k_R+2}(\uvkd-\uv_{k-1}^\delta) \,,
        \qquad
        k_R = \operatorname{ceil}(k/R) \,,
        \\ 
        \uvkpd &=  T\kl{ \zv_k^\delta  + \okdt \Mv^{-1} \Hv_\kmod^T \Nv^{-1} \kl{ \wv_\kmod^\delta -  \Hv_\kmod  \zv_k^\delta  }} \,,
    \end{aligned}
    \end{equation} 
where $\okdt$, $\Mv$, $\Nv$, $\Hv_r$, and $\wv_r^\delta$, are defined as in \eqref{stepsize_general_discrete}, \eqref{def_matrices_M_N_Hr_alternative}, and \eqref{def_Hr_wrd}, respectively.
\end{PNKR}

As mentioned in Section~\ref{modifications}, the order of the equations induced by the index $k_R$ in the PNKR method \eqref{PNKR} does not necessarily have to be circular. Hence, in the numerical examples considered below we replace $k_R$ by a random permutation of $\Kl{1,\dots,R}$.

\subsection{The reduced PNKR method}

In the PNKR method \eqref{PNKR} it is advantageous from the computational point of view to keep the regularity of the functions $\psi_n$ as low as possible, such that the structure of the matrices $\Nv$ and $\Hv_r$ remains as simple as possible. On the other hand, when deriving the PNKR method we assumed that $\psi_n \in H^s(\Omega)$ in order to use the underlying structure of the problem to our advantage. At first, these two observations seem to contradict each other. However, we shall now see that this need not necessarily be the case.

For this, we first consider the special case $s=0$. In addition, we limit our attention to the typical situation that the domain $\Omega$ is rectangular, noting that the subsequent arguments can also be generalized to more complicated domains. In this case, we have that $H^s(\Omega)= \LtO$ and thus for the PNKR method the functions $\psi_n$ can be chosen as piecewise constant functions on a uniform rectangular subdivision of $\Omega$. Since also the domain $\Theta$ is rectangular, the same is true for the functions $\varphi_l$, and thus we get
\begin{equation}\label{def_matrices_M_N_Hr_L2_uniform}
    \begin{split}
        \Mv := c_M \Iv \,,
        \qquad
        \Nv := c_N \Iv \,,
        \qquad
        \Hv_r := \kl{ \delta_{j,n(i)}
        \kl{\int_\Theta  k(\theta,\lambda_r)  \varphi_{l(i)}(\theta) \, d\theta }  }_{i,j=1}^{M,N} \,,
    \end{split}
    \end{equation}
where $\Iv$ denotes the identity matrix and the constants $c_M$ and $c_N$ depend on the chosen subdivisions of $\Omega$ and $\Theta$. With this, the second line of \eqref{PNKR} then takes the form
    \begin{equation}\label{LK_discrete_combined_L2}
        \uvkpd = T\kl{ \uvkd + \okdt \kl{c_M c_N}^{-1} \Hv_\kmod^T \kl{ \wv_\kmod^\delta -  \Hv_\kmod  \uvkd } } \,,
    \end{equation}
which is particularly simple to implement. Next, consider the case $s > 0$. From the purely implementational point of view, the reason why piecewise constant functions can now no longer be used in \eqref{PNKR} is that then the matrix $\Mv$ defined in \eqref{def_matrices_M_N} is not well-defined any more. However, note that at least formally there holds \cite[Proposition~2.1]{Neubauer_1988}
    \begin{equation*}
        \Mv \overset{\eqref{def_matrices_M_N}}{=} 
        \kl{ \spr{\phi_i,\phi_j}_\HsOT  }_{i,j=1}^M 
        =
        \kl{ \spr{\phi_i,(E_s^*)^{-1}\phi_j}_\LtOT  }_{i,j=1}^M \,,
    \end{equation*}
where as before $E_s : \HsOT \to \LtOT$ denotes the embedding operator. This indicates that the matrix $\Mv^{-1}$ in \eqref{PNKR} can be understood as the discrete analog of the adjoint embedding operator $E_s^*$. This operator essentially behaves like a smoothing operator, and its application can e.g.\ be understood as a damping of higher order Fourier components of a function, or as a convolution with a suitable spatial filter \cite{Hubmer_Sherina_2022} in dependence on the index $s$ of the Sobolev space $\HsOT$. In our discrete setting, the action of $\Mv^{-1}$ on a vector $\uv$ can thus be approximated by the operator
    \begin{equation}\label{def_Zs}
        Z_s : \R^M \to \R^M \,,
        \qquad
        \uv = \Kl{u_i}_{i=1}^M \mapsto \Kl{(\uv \ast \zv_s)(i)}_{i=1}^M \,,
    \end{equation}
where the convolution-kernel $\zv_s$ is suitably chosen depending on $\HsOT$. Hence, one may use piecewise constant basis functions in the PNKR method also for the case $s > 0$ by replacing the application of $\Mv^{-1}$ with the application of $Z_s$. Together with the specific form of the matrices $\Nv$ and $\Hv_r$ given in \eqref{def_matrices_M_N_Hr_L2_uniform} we thus arrive at

\begin{SPNKR}\label{method_SPNKR}
Let $s \in \R$, $\Omega := [x_1^{\min},x_1^{\max}]\times[x_2^{\min},x_2^{\max}] \subset \R^2$, and let $\Theta \subset \R^3$ be defined as in \eqref{def_Theta}. Furthermore, let $\Kl{\psi_n}_{n=1}^N$ and $\Kl{\varphi_l}_{l=1}^L$ denote two sets of piecewise constant functions on uniform regular subdivisions of $\Omega$ and $\Theta$, respectively. Moreover, let the finite dimensional spaces $\XM$ and $\YN$ be defined as in \eqref{def_XN_YM}, and for all $r=1,\dots,R$ let $K_r$ be defined as in \eqref{def_operator_K_r} and let $\yd_r \in Y$. Then the \emph{reduced PNKR method} for solving the galactic archaeology Problem~\ref{prob_Kv} is given by
    \begin{equation}\label{rPNKR}
    \begin{split}
        \zv_k^\delta &= \uvkd + \frac{k_R-1}{k_R+2}(\uvkd-\uv_{k-1}^\delta) \,,
        \qquad
        k_R = \operatorname{ceil}(k/R) \,,
        \\ 
        \uvkpd &= T\kl{ \uvkd + \okdt \kl{c_N}^{-1} Z_s \kl{ \Hv_\kmod^T  \kl{ \wv_\kmod^\delta -  \Hv_\kmod  \uvkd }} } \,,
    \end{split}
    \end{equation}
where $\wv_r^\delta$, $\okdt$, $\Hv_r$, $Z_s$ are defined as in \eqref{def_Hr_wrd}, \eqref{stepsize_general_discrete}, \eqref{def_matrices_M_N_Hr_alternative}, \eqref{def_Zs} respectively, and the constant $c_N$ denotes the size of any element of the uniform subdivision of $\Omega$.
\end{SPNKR}

As before, the index $k_R$ in the reduced PNKR method \eqref{rPNKR} can be replaced by a random permutation of $\Kl{1,\dots,R}$. Note that the reduced PNKR method is computationally much less demanding than the full PNKR method \eqref{PNKR}. Furthermore, note that the convolution operator $Z_s$ can also be replaced by a suitable damping in the discrete Fourier or wavelet domain. However, due to its approximate nature with $\Mv^{-1}$ being replaced by $Z_s$ we suggest to use the reduced PNKR method only if the computational effort for running the full PNKR method \eqref{PNKR} is practically infeasible.

\section{Numerical examples}\label{sect_numerical_examples}

To test our reconstruction approach, we apply it to a dataset simulated from a known ground truth density. In Sections~\ref{ssec_discretization}-\ref{ssec_implementation}, we describe the dataset used, as well as details about the discretization and implementation. In Section~\ref{ssec_results}, we present results of our reconstruction algorithm applied to our test problem using two separate sets of basis functions. We  show that our method performs well in a quantitative sense in both cases. To further validate our method, in Section~\ref{ssec_ppxf} we compare our recoveries to those obtained using the widely-used Penalized Pixel-Fitting (pPXF) \cite{Cappellari17} spectral fitting software, and compare qualitative aspects of the obtained recoveries.

\subsection{Visualization}\label{ssec_vis}

To visualize the $5$-dimensional density $f(\xv,\theta)$, we will display several projections which are commonly used in galaxy IFU analysis see e.g.\ \cite{guerou16}. These comprise maps of mean metallicity and age, and maps of several statistics describing the velocity distribution. In addition, we will explicitly show a selection of velocity distributions evaluated at a selection of positions. These quantities are mathematically defined below, and the visualization of the ground truth is shown in Figure~\ref{fig_moments}.

To visualize the age and metallicity variations, we show their first order moments as functions of $\xv$. We first define some useful marginal probability distributions of the target distribution function $f$,
    \begin{equation*}\label{marginals}
    \begin{split}
        M := \int_\Omega\int_\Theta f(\xv,\theta) \, d\xv \, d\theta \,, 
        &\qquad 
        p(\xv) := \frac{1}{M} \int_\Theta f(\xv,\theta) \, d\theta \,, 
        \\
        p(\xv,z) := \frac{1}{M} \int_{v_\mathrm{min}}^{v_\mathrm{max}} \int_{t_\mathrm{min}}^{t_\mathrm{max}} f(\xv, \theta) \, dv \, dt \,, 
        &\qquad
        p(\xv,t) := \frac{1}{M} \int_{v_\mathrm{min}}^{v_\mathrm{max}} \int_{z_\mathrm{min}}^{z_\mathrm{max}} f(\xv, \theta) \, dv \, dz \,.
    \end{split}
    \end{equation*}
The spatial variations of the mean age and metallicity are then given by
    \begin{equation*}
        \mu_z(\xv) := \int_{z_\mathrm{min}}^{z_\mathrm{max}}  \frac{z\, p(\xv,z)}{p(\xv)} \, dz \,,
        \qquad
        \text{and}
        \qquad
        \mu_t(\xv) := \int_{t_\mathrm{min}}^{t_\mathrm{max}}  \frac{t \, p(\xv,t)}{p(\xv)} \, dt \,.
    \end{equation*}
To visualize marginal velocity distributions, we will show a slightly different quantity in order to later facilitate comparison with existing methods (see Section~\ref{ssec_ppxf}). The target density $f$ as we have defined it has units of mass per volume element, i.e. integrated over the whole domain it will give the total stellar mass of a galaxy (see equation \ref{eq_mass_normalisation}). A related quantity is the \emph{light-weighted} distribution $f_\mathrm{LW}$, which has units of luminosity rather than mass. These are related via
    \begin{equation*}
        f_\mathrm{LW}(\xv, v,z,t) := f(\xv, v,z,t) \int_{\lambda_\mathrm{min}}^{\lambda_\mathrm{max}}S(\lambda ; t,z) \, d\lambda \,,
    \end{equation*}
where $S$ are SSP templates in equation \eqref{eq_main_general}. As in \eqref{marginals}, we can define the light-weighted marginal distributions $p_\mathrm{LW}(\xv)$ and $p_\mathrm{LW}(\xv, v)$, which finally giving us light-weighted velocity distribution at a position $\xv$, i.e.,
    \begin{equation}
        p_\mathrm{LW}(v|\xv) := \frac{p_\mathrm{LW}(\xv, v)}{p_\mathrm{LW}(\xv)} \,.
    \end{equation}
We will show maps of four quantities summarising these velocity distributions: their mean $\mu_v$, standard deviation $\sigma_v$ and two higher order statistics $h_3$ and $h_4$. These are coefficients of Gauss-Hermite expansions, which are commonly used to describe velocity distributions in galactic astronomy \cite{vdmarel_franx_93}, and roughly correspond to skewness and kurtosis, respectively. In addition to these four quantities, we will also show a selection of velocity distributions $p_\mathrm{LW}(v|\xv_i)$ explicitly evaluated at $9$ locations $\xv_i$. These locations have been chosen to maximise the variety of distributions on display.

\subsection{Ground truth and discretization}
\label{ssec_discretization}

The ground truth density $f$ comes from a simple descriptive model of galaxy properties and the dataset is simulated with parameters comparable to existing, real IFU galaxy datasets of medium-size and high spectral quality. The specific choices are described in detail below but, in overview, they provide a reasonably realistic yet idealized and computationally-manageable arena for developing a novel reconstruction technique. 

\begin{figure}[ht!]
    \centering
    \includegraphics[width=\textwidth]{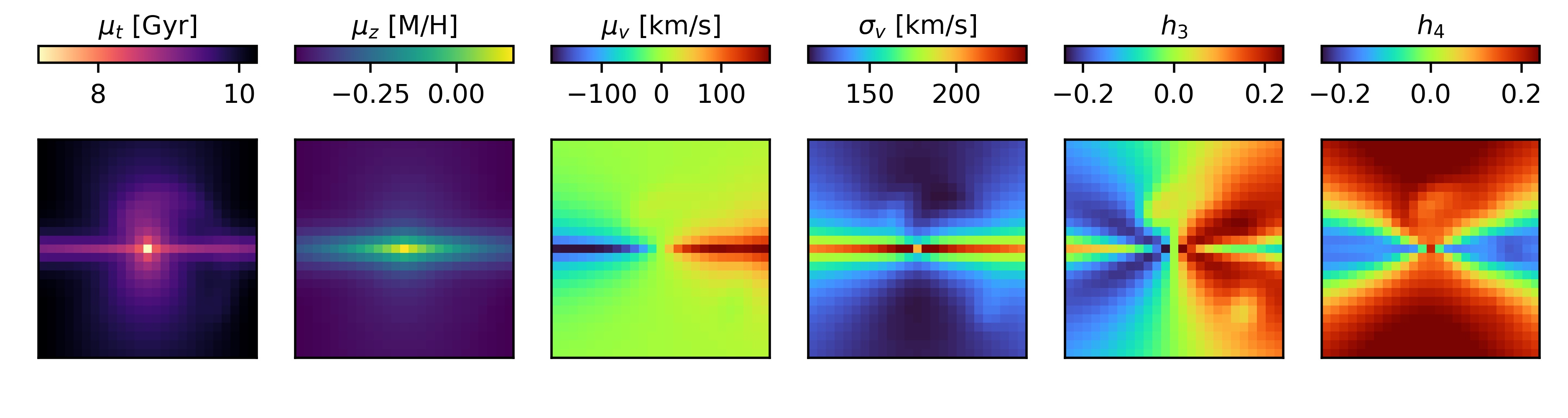} \\   
    \includegraphics[width=\textwidth]{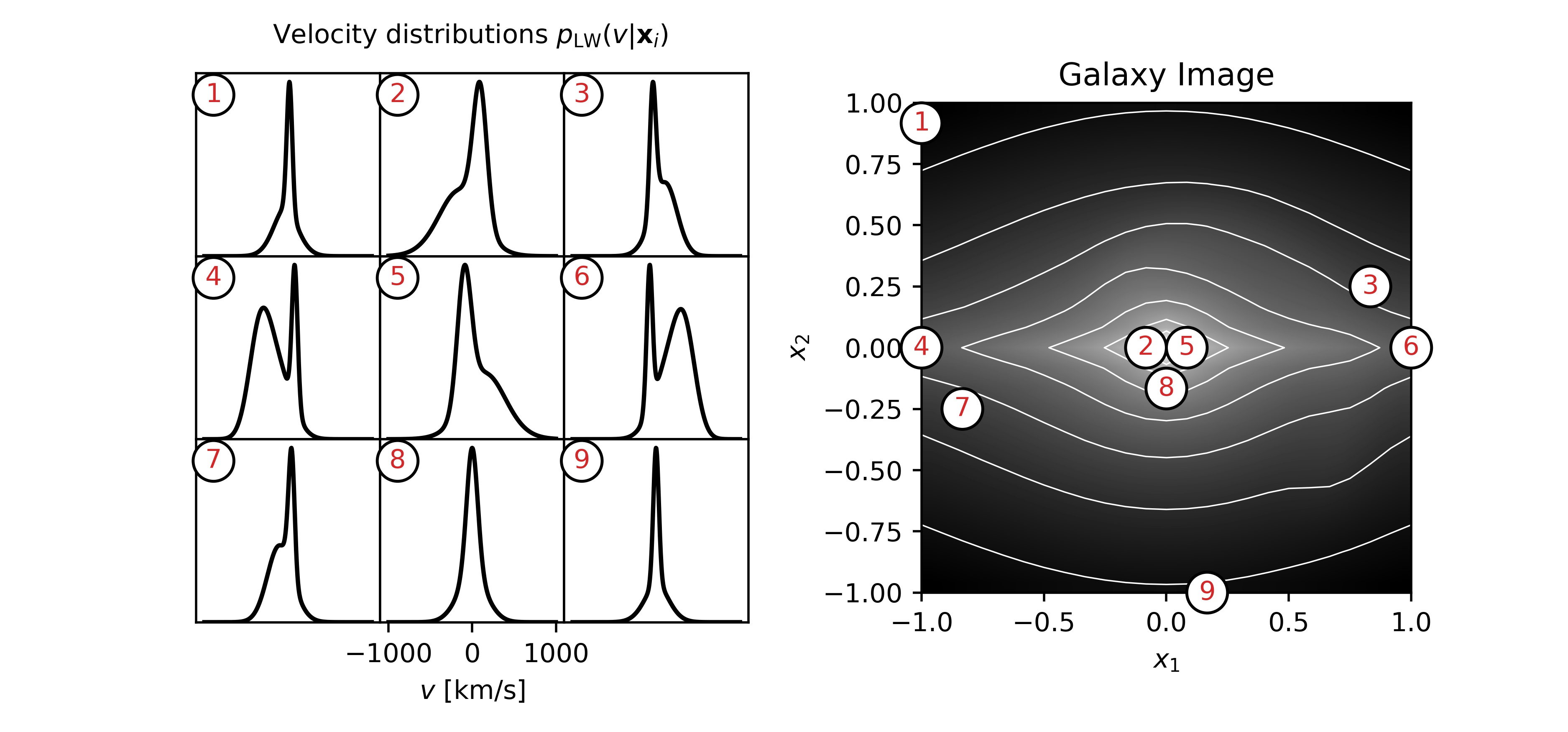} \\   
    \caption{Visualisation of the ground truth density. Top row shows maps of mean age $\mu_t$, mean metallicity $\mu_z$, followed by four statistics describing the velocity distribution: mean $\mu_v$, standard deviation $\sigma_v$, and Gauss Hermite expansion coefficients $h_3$ and $h_4$. The bottom row shows a grid of velocity distributions $p_\mathrm{LW}(v|\xv_i)$ evaluated at the positions $i=1,...,9$ highlighted on the galaxy image. See Section~\ref{ssec_vis} for details. The galaxy consists of two superposed, counter-rotating disks, giving rise to the bimodal velocity distributions particularly prominent in panels $i=4$ and $6$, and a weak stream-like component accounting for the protruding contours in the bottom-right of the image and the blobs in the $h_3$ distribution.
    }
    \label{fig_moments}
\end{figure}

The ground truth is a mixture of three physical galactic components. The dominant component is a thin, fast-rotating disk which contributes $70\%$ of the total stellar mass, while a thicker, slower-rotating disk contributes $29\%$. The two disks are counter-rotating, producing the highly skewed and bimodal velocity distributions seen in Figure~\ref{fig_moments}. The remaining $1\%$ of stars are in a stream-like component which is responsible for some of the irregular features seen in Figure~\ref{fig_moments}, such as the slightly protruding contours, and blobs in the $h_3$ map. Each component is specified with independent star-formation histories, chemical-enrichment prescriptions from \cite{Zhu20}, and stellar-age dependent morphologies and kinematics. Code to evaluate this model is available online \cite{Popkinmocks}, while a complete description of the astrophysical ingredients is in preparation. The dataset is stored as both a noise-free version and a noisy version created by adding pixel-wise noise sampled from a Gaussian distribution with standard deviation of 1\% of the noise-free value. This results in an overall noise level of $\delta = 1\%$.

We next define the discretization of our domain. The discretization of $\Theta$ consists of a uniform subdivision of $[-1000,1000]$ with $27$ grid points for $v$, a nonuniform subdivision of $[-2.66, 0.36]$ with $7$ grid points for $z$ and a nonuniform subdivision of $[0.015, 14.25]$ with $19$ grid points for $t$. The choices in $z$ and $t$ are dictated by our choice of SSP templates $S$, for which we use a subset of the MILES models \cite{Vazdekis_SanchezBlazquez_FalconBarroso_Cenarro_Beasley_Cardiel_Gorgas_Peletier_2010}. The SSP template spectra $S$ are sampled at $687$ different wavelengths $\lambda$ between $4800$nm and $5700$nm with a velocity resolution of $75$km/s, similar to the Fornax 3D IFU survey \cite{Sarzi18} which uses the MUSE IFU instrument. This velocity resolution in turn sets our discretization in $v$. For the discretization of $\Omega$ we have a uniform grid of $[-1,1]^2$ with $26$ discretization points in each dimension.

\subsection{Choice of basis functions}\label{ssec_basis}

We perform numerical experiments using two different sets of basis functions. First, to test the procedure with minimal computational cost, we use a basis of piecewise constant functions. This corresponds to choosing $s=0$ in \eqref{PNKR}. In this case, we have one basis function per element of the discretization, with the basis function being defined as constant $1$ within the element, and $0$ everywhere else. Thus, the amount of basis functions is $N:=625$ for $\Omega$ and $L:=2808$ basis functions for $\Theta$. The resulting mass matrix $\Mv$ is a diagonal matrix, making the application of $\Mv^{-1}$ very simple without involving any expensive matrix-vector multiplication.  

In a next step, to incorporate the expected smoothness of the ground truth, we test the procedure using piecewise linear basis functions, corresponding to $s=1$ in \eqref{PNKR}. The basis functions are defined as the tensor product of 1D FEM hat functions. In this case, we have one basis function per node of the discretization. In order to retain the same dimension of the problem as in the piecewise constant case, we define a new discretization, where each node is the average of two neighboring nodes of the original discretization. This again results in $N:=625$ basis functions for $\Omega$ and $L:=2808$ basis functions for $\Theta$. The resulting mass matrix $\Mv$ has a sparse block-tridiagonal structure (Figure~\ref{fig_M_H1}) with a length of $N\cdot L=1.755 \cdot 10^6$.  

\begin{figure}[ht!]
    \centering
    \includegraphics[scale=0.16]{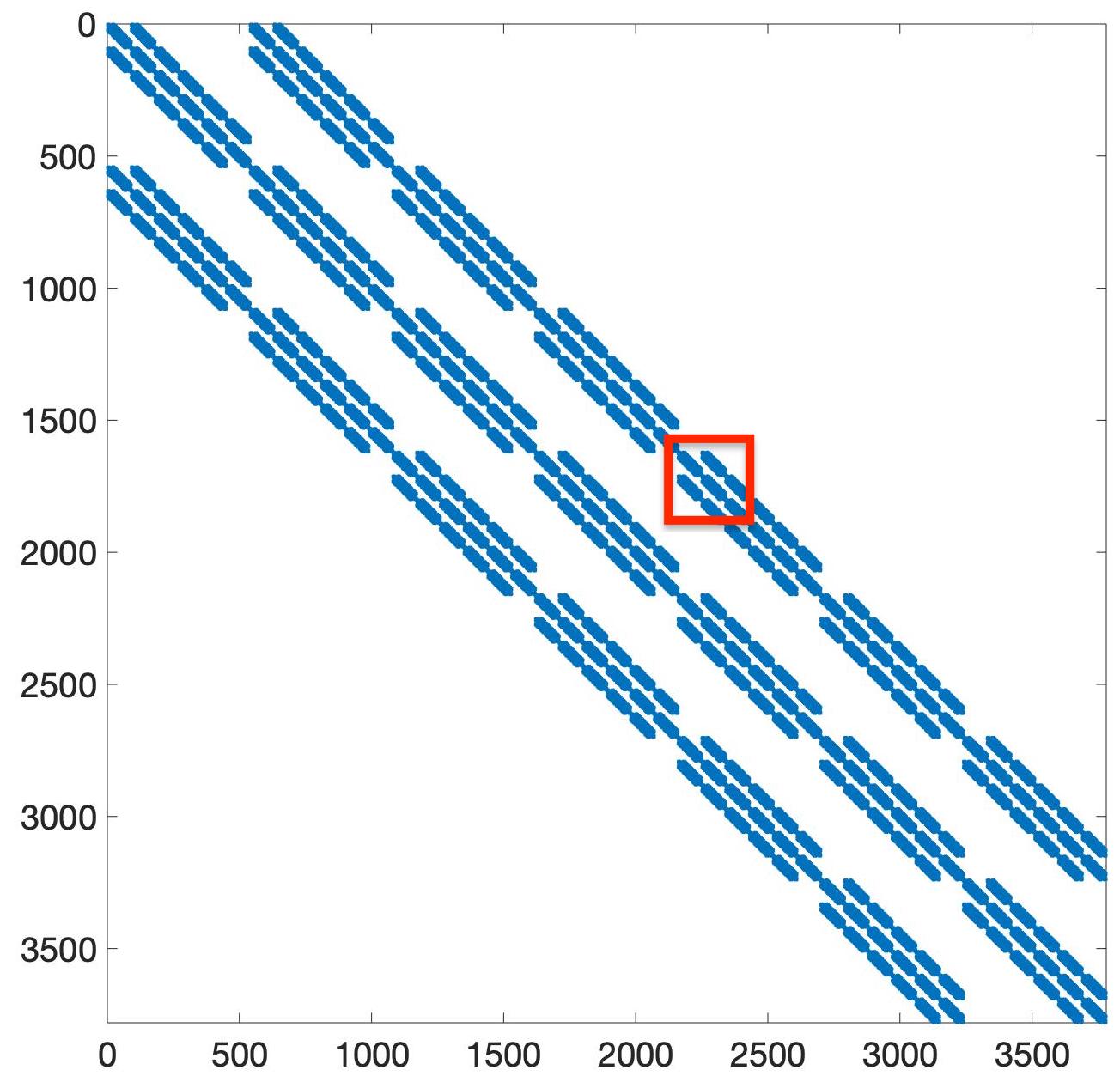} \includegraphics[scale=0.16]{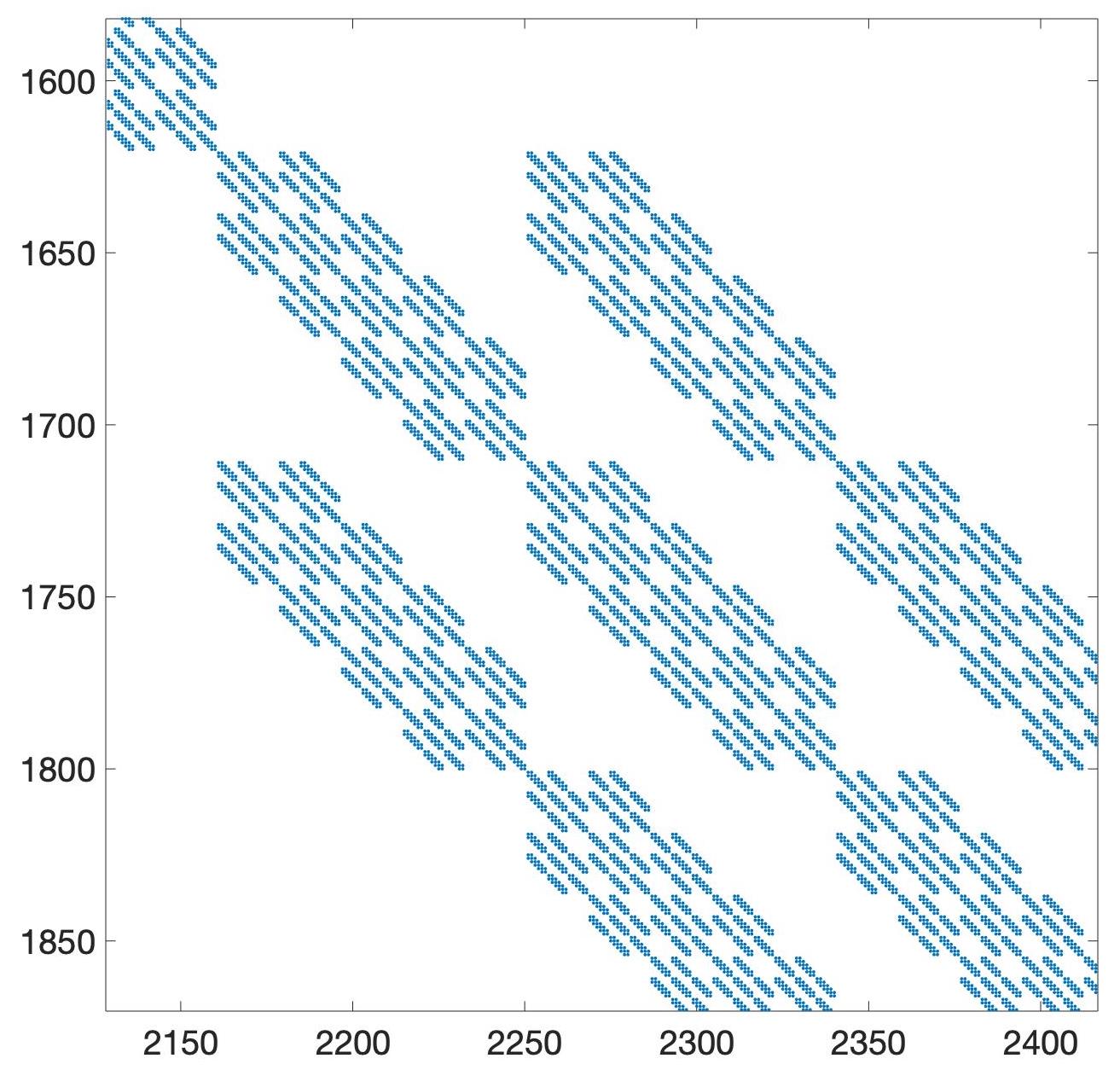}
    \caption{Block-tridiagonal structure of the mass matrix $\Mv$ for piecewise linear basis functions. For visualization purposes, the dimension has been reduced to $N=18$ and $L=210$ for this figure. (a) illustrates the structure of the entire matrix, whereas (b) shows additional detail in the zoomed in area within the red square in (a)}
    \label{fig_M_H1}
\end{figure}

\subsection{Implementation and computational aspects}\label{ssec_implementation}
Computations were carried out in Matlab on a desktop computer with an Intel Xeon E5-1650 processor with $3.20$ GHz and $16$ GB RAM. The most expensive part is computing $\Mv^{-1}$ at every step $k$ in  \eqref{PNKR}. As is standard in dealing with large-scale sparse matrices, we do not apply the inverse directly, but compute a solution of the forward operator problem 
    \begin{equation}\label{forward} 
        \Mv \xv = \zv_k^\delta \,,
        \qquad 
        \zv_k^\delta = \Hv_\kmod^T \Nv^{-1} \kl{ \wv_\kmod^\delta -  \Hv_\kmod  \uvkd } \,.
    \end{equation} 
The large scale nature of the problem renders assembly of the full matrix $\Mv$ infeasible. However, taking note of the Kronecker structure of $\Mv$, we can represent the matrix as $\Mv = \Psi \otimes \Phi \,,$ where $\otimes$ denotes the Kronecker product and
    \begin{equation}\label{Phi_Psi} 
        \Psi :=\kl{ \spr{ \psi_i, \psi_j }_{H^1(\Omega)}}_{i,j=1}^N \,,
        \qquad
        \text{and} 
        \qquad
        \Phi :=\kl{ \spr{ \varphi_i, \varphi_j }_{H^1(\Theta)} }_{i,j=1}^L \,.
    \end{equation} 
This structure allows multiplication with a vector to be implemented in a very efficient way, where only the much smaller matrices $\Phi$ and $\Psi$ need to be assembled; see, e.g., \cite[Section 12.3]{Golub_VanLoan_2013} for more information about the mathematics of Kronecker-structured systems. Similarly, the inverse of $\Mv$ can be represented as $ \Mv^{-1} = \Psi^{-1}\otimes \Phi^{-1} \,,$ where $\Psi^{-1}$ and $\Phi^{-1}$ can be computed in a pre-processing step, for example using LU decomposition. The application of $M^{-1}$ can now be performed very efficiently. Using Matlab notation, $\Mv^{-1} z_k = \verb+reshape+ ( \Psi^{-1} \verb+reshape+(z_k, N,L) \Phi^{-1}, N \cdot L,1)$. Note that since $\Psi^{-1}$ and $\Phi^{-1}$ are full matrices, one application of $\Mv^{-1}$ is much more expensive than one application of $\Mv$. 

Regarding the stepsize and stopping criterion, we use \eqref{stepsize_general_discrete} with $\tau=1.2$ and constant $\alpha_k^\delta$. We use values of $\alpha_k^\delta = 10$ in the case of piecewise constant basis functions, and $\alpha_k^\delta =1000$ in the case of piecewise linear basis functions. In  Figure~\ref{compare_h5}, we show results for different values of the parameter $\beta$ in \eqref{beta}. In addition to the default value $\beta=1$, we show reconstructions for $\beta=0.5$ and $\beta=0.01$, corresponding to a lower weight placed on the first order derivative.

The inner products in \eqref{Phi_Psi} are calculated using the trapezoid rule. We divide  the support of the integrand, i.e. $\operatorname{supp}(\psi_i) \, \cap \,  \operatorname{supp}(\psi_j) $ (and analogously $\operatorname{supp}(\varphi_i) \, \cap \,  \operatorname{supp}(\varphi_j) $) using $50$ grid points within each dimension with equidistant spacing. The right hand sides $\wv_k^\delta$ in \eqref{def_Hr_wrd} are also computed with trapezoid integration. 

Looping through the equations in a random ordering has been reported to speed up convergence \cite{Strohmer_2006}. Therefore, after every full loop (i.e. after every $R$ iterations), we create a random permutation of $(1,...,R)$ using the Matlab function $\verb+randperm+(R)$ and then loop through the $R$ equations in the ordering specified by the permuted tuple. We start the iteration with the initial guess $u_0=0$. After every full loop, we calculate the residual and the relative error.
    \begin{equation*}
        \operatorname{res}_k^2 = 
        \sum_{j=1}^L \norm{ \Hv_{j} (u^*-u_k^\delta) }^2 \,,
        \qquad
        \text{and}
        \qquad
        \operatorname{error}_k = \frac{ \norm{ u^*-u_k^\delta }}{\norm{u^*}} \, , 
    \end{equation*}
    where $u^*$ denotes the representation of the ground truth $f$ in terms of basis functions, i.e. $f = \sum_{i=1}^M u^*_i \phi_i$.
We summarize the procedure in Algorithm~\ref{algorthim_pnkr}.
\begin{algorithm*}[ht!]
\caption{Projected Nesterov Kaczmarz Reconstructor (\textbf{PNKR})}
\label{algorthim_pnkr}
\vspace{1mm}
\textbf{Input} data $y^\delta_r$ for $r=1...R$,  noise level $\delta \in \R^R$, wavelengths $\lambda_1, \cdots, \lambda_R$ and corresponding SSP spectra $S$, discretization of $\Omega$ and $\Theta$
\begin{algorithmic}
\State - select piecewise constant $(s=0)$ or piecewise linear $(s=1)$ basis functions $\{\varphi_i \}_{i=1}^L$ and $\{\psi_j \}_{j=1}^N$. 
\State - compute vectors $\wv_r^\delta = \langle \yv_r^\delta, \psi_j \rangle_{j=1}^N $
\State - Implement function representations of $\Nv$, $\Mv$ and $\Hv_r$ for $r=1...R$ defined by \eqref{def_matrices_M_N_Hr_alternative}
\State - Set $\uv_0 = 0 \in \R^M$, $k_R=1$, $k=0$
\vspace{2mm}
\While{discrepancy principle \eqref{stepsize_general_discrete} is not fulfilled for all equations}
    \State $P= \verb+randperm+(R)$
    \For{$\hat{r}=1:R$} 
    \State set index $r = P_{\hat{r}}$
    \State check discrepancy principle: 
        \If{$\left( \Nv^{-1} \kl{ \wv_r^\delta -  \Hv_r  \uvkd }\right)^T \Nv \Nv^{-1} \kl{ \wv_r^\delta -  \Hv_r  \uvkd } > 1.2 \delta_r$}
        \State - apply Nesterov acceleration:
        \State $ \zv_k^\delta = \uvkd + \frac{k_R-1}{k_R+2}(\uvkd-\uv_{k-1}^\delta) $
        \State - compute or choose stepsize $\okdt$
        \State - compute new iterate and apply thresholding: 
        \State $ \uvkpd =  T(\zv_k^\delta  + \okdt \Mv^{-1} \Hv_r^T \Nv^{-1} \kl{ \wv_r^\delta -  \Hv_r  \zv_k^\delta  })$
        \State - increment iteration index 
        \State $k=k+1$
        \EndIf
\EndFor
\State $k_R = k_R +1 $
\EndWhile
\vspace{2mm}
\State - calculate density $f$ from coefficient vector $u_k$ by multiplying with basis functions
\State  $f = \sum_{i=1}^M u_{k,i} \psi_{n(i)} \varphi_{l(i)}$
\end{algorithmic}
\vspace{2mm}
\textbf{Output}: reconstructed 5D density $f$
\end{algorithm*}

\subsection{Numerical results}\label{ssec_results}
\begin{figure}[ht!]
    \centering
        \includegraphics[width=0.45\textwidth]{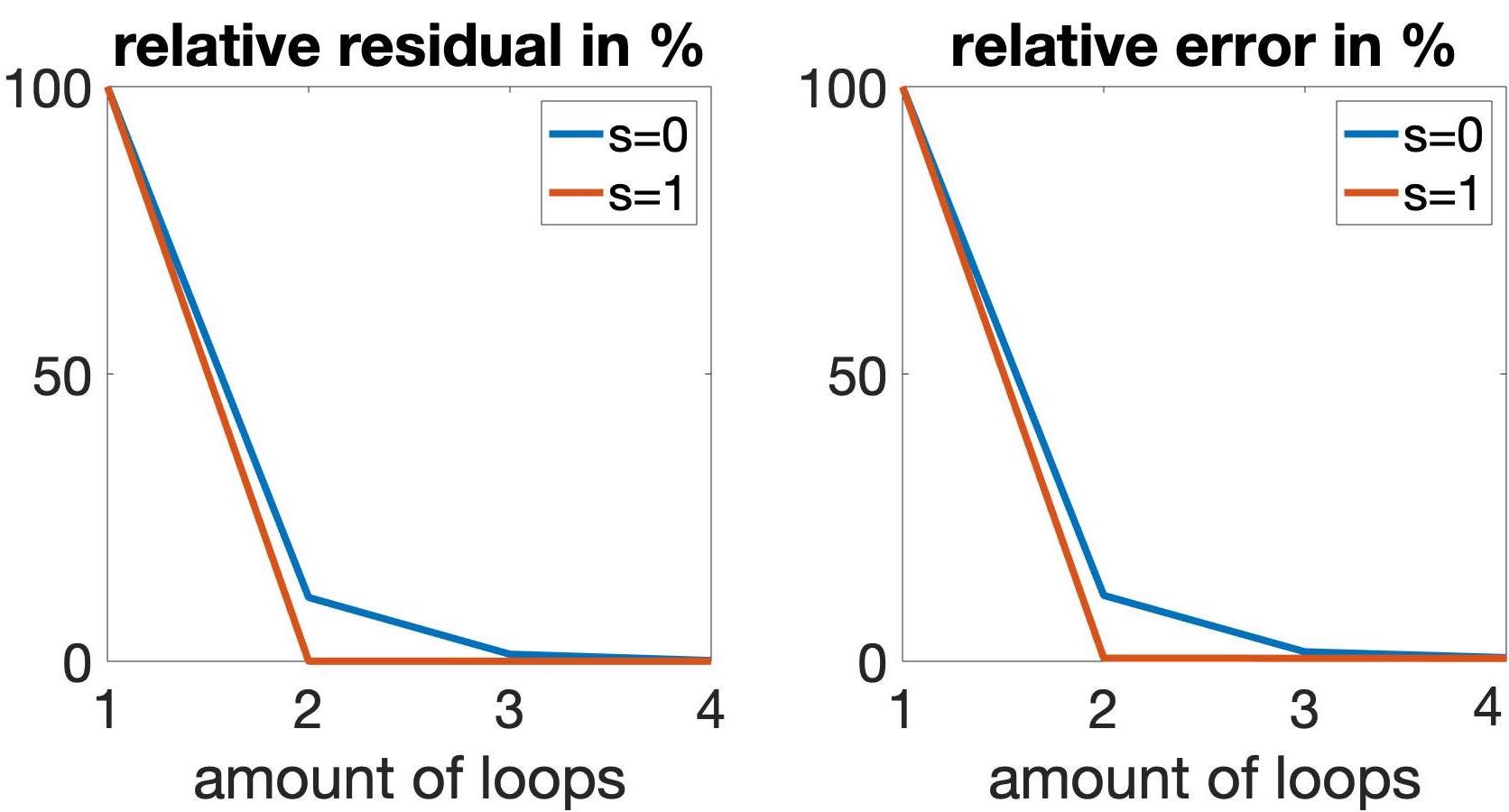} \hspace{2mm}
        \includegraphics[width=0.45\textwidth]{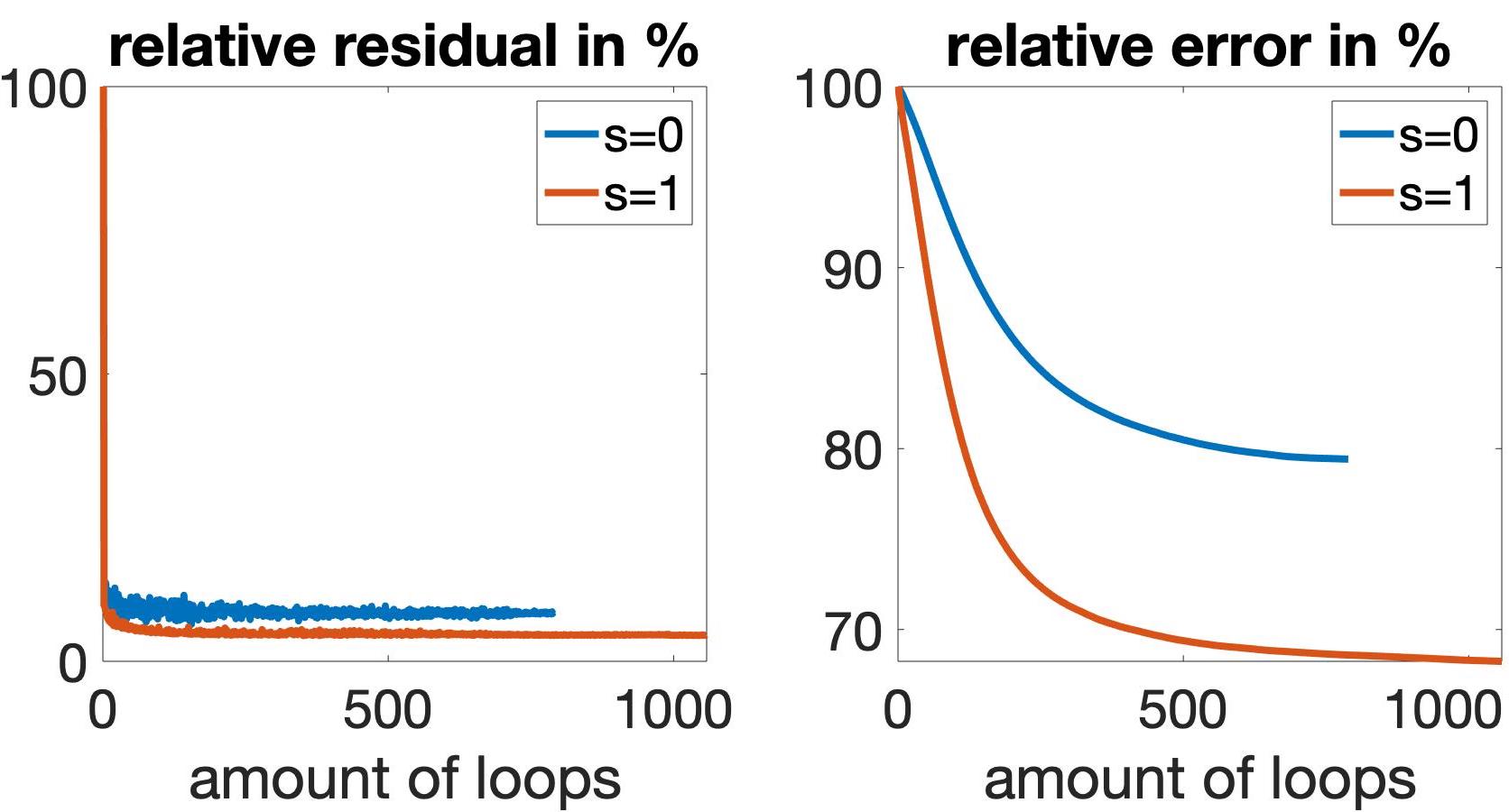}
    \caption{First and second panel: Relative residual and relative error measured from PNKR being applied to the ground truth $u_\perp^*$, i.e. without null space components in the solution and with noise-free data being known. Third and fourth panel: Relative residual and relative error measured from PNKR being applied to the ground truth $u^*$ with a noise level of $1\%$ in the data. The iteration was stopped using the discrepancy principle \eqref{stepsize_general_discrete}.}
    \label{reconstructions}
\end{figure}
We perform the PNKR method \eqref{method_PNKR} using piecewise constant ($s=0$) and piecewise linear ($s=1$) basis functions as described in Section~\ref{ssec_basis}. To validate our method, we first project the ground truth to the orthogonal complement of the null space of the operator, i.e., $u^*_\perp = \sum_{j=1}^L \Hv^T_j \Hv_j u^*$. In addition, we assume noise-free data to be known. In the piecewise constant case, within four loops through all wavelengths the relative error reaches $0.5\%$ and the relative residual reaches $0.013\%$. In the piecewise linear case, similar values are achieved after just one loop (Figure~\ref{reconstructions}). This indicates that the method is able to retrieve the correct density if no null space components are present. However, we observed that in general the null space of the problem is rather large, which means that components outside this null space cannot be recovered. 

We now apply PNKR to the ground truth $u^*$ without projection. We consider the more realistic scenario where the data  $y^\delta$ is impacted by the presence of Gaussian noise with a relative noise level of $1\%$. We show the measured relative residuals and errors in Figure~\ref{reconstructions}. In the case of piecewise linear basis functions, the procedure stops after $1058$ loops. The residual falls gradually to about $4.7\%$, while relative errors remains at $68.2\%$. In the case of piecewise constant basis functions, the iterations stops after $720$ loops. The residual falls  to about $8.3\%$, while relative errors remains at $79.5\%$. The high relative errors indicate that significant null space components are present in the ground truth $u^*$. 

Due to the large scale nature of the problem, PNKR requires significant computational resources. Using piecewise constant basis functions, one loop through all equations takes about $54s$. The total elapsed time until the iteration stopped amounted to about $4$ hours.  For piecewise linear functions, this increases to $16$ minutes for one loop and $48$ hours in total.

\subsection{Comparison with pPXF}\label{ssec_ppxf}

To validate our method we compare our recoveries against those measured using the
Penalized Pixel-Fitting (pPXF) software \cite{Cappellari17}, a mature and popular tool for modelling galaxy IFU datacubes. This software can model stellar population-kinematics (in addition to other effects such as gas emission and dust absorption), and uses flexible, non-parametric models of the velocity distribution. However, one limitation of pPXF is that it does not allow for the full datacube to be fitted simultaneously, instead modelling its constituent spectra independently. While existing datacube modelling tools do exist \cite{davis13, di_teodoro_15, Bouche_15, Varidel19}, none of these are suitable for combined modelling of stellar-populations and complex velocity distributions. Hence, we perform a baseline comparison only against pPXF.

We now summarise the relevant aspects of the pPXF forward model. Whereas we attempt to recover the joint distribution function $f$, pPXF makes the simplifying assumption that $f$ is separable, with the specific factorisation from \eqref{eq_factor_joint}. Inserting this factorisation into \eqref{eq_main_general}, gives the pPXF forward model
    \begin{equation}\label{eq_ppxf}
    \begin{split}
        f_1(\xv_i) \int_{\vmin}^{\vmax} \frac{1}{1 + v/c} f_2(v | \xv_i) \kl{ \int_{\zmin}^{\zmax} \int_{\tmin}^{\tmax}  S\kl{\frac{\lambda}{1 + v/c} ; t, z}  f_3(t, z| \xv_i) \, dt \, dz  } \, dv
        = y_i(\lambda) \,.
    \end{split}
    \end{equation}
Compared to our forward model in \eqref{eq_main_general}, the right-hand side is now a collection of 1D spectra $y_i(\lambda)$ at positions indexed by $i$, rather than a 3D datacube $y(\lambda,\xv)$. The specific form of the population distribution $f_3(t, z| \xv_i)$ used by pPXF are piecewise constant values on a rectangular grid in $t$ and $z$, with Tikhonov regularisation used to promote smoothness in the recovery. The velocity distributions $f_2(v | \xv_i)$ in pPXF are as Gauss-Hermite expansions with parameters $\mu_v, \sigma_ v, h_3, h_4, \dots , h_N$, with regularisation (known as bias in pPXF parlance) to encourage the non-Gaussian parameters $h_i$ to remain small. Note that this velocity distribution is necessarily light-weighted, not mass-weighted (see the discussion in Section~\ref{ssec_vis}).

\begin{figure}[ht!]
    \centering
    \includegraphics[width=\textwidth]{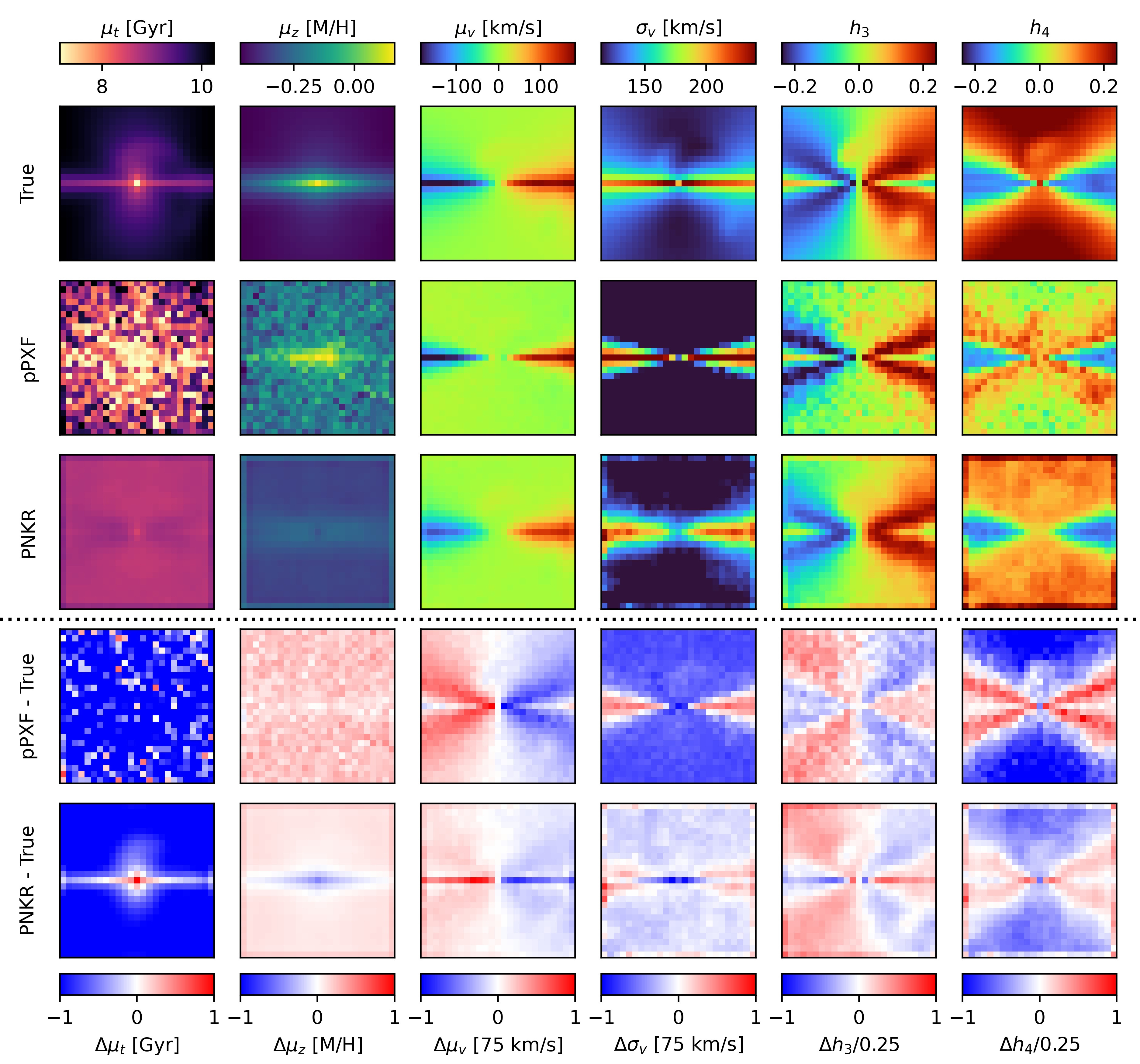}
    \caption{Comparison of recoveries using pPXF and PNKR. Columns show the same quantities as in Figure \ref{fig_moments}. The top three rows show the ground truth, the pPXF and PNKR recoveries. The bottom two rows show the error between the recoveries and the ground truth. Whilst neither pPXF nor PNKR accurately reconstruct population maps $\mu_t$ and $\mu_z$, for the kinematics (four right-most columns) PNKR achieves an overall more accurate reconstruction.
    }
    \label{compare_moments}
\end{figure}

\begin{figure}[ht!]
    \centering
    \includegraphics[width=\textwidth]{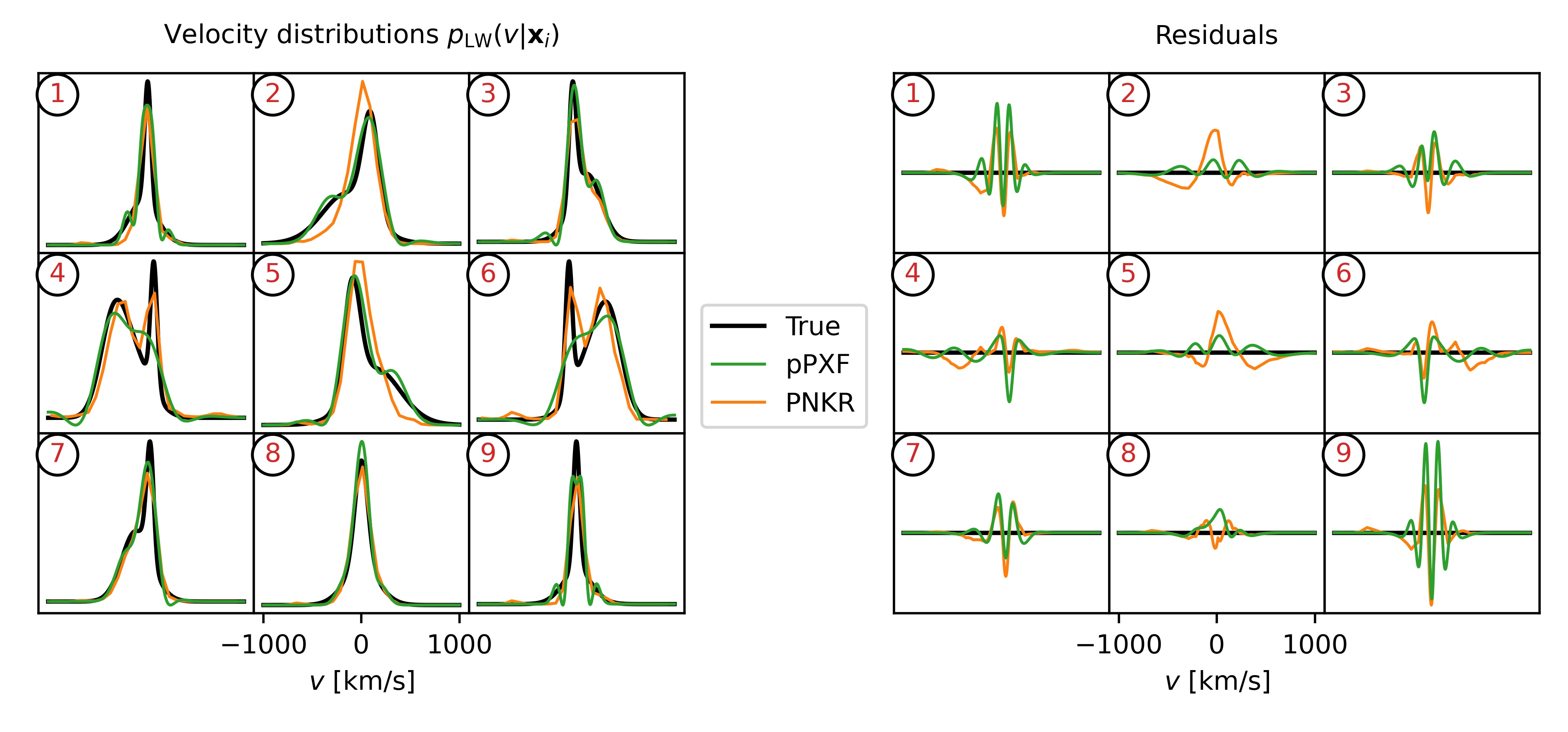}
    \caption{
    Comparison of recovered velocity distributions using pPXF and PNKR. The $i=1,..,9$ distributions are the same as in Figure \ref{fig_moments}. The left grid shows the distributions, the right shows residuals from the ground truth. Overall, the amplitude of the PNKR residuals are smaller then those from pPXF. Panels 4 and 6 show that PNKR can produce bimodal distributions, while the Gauss-Hermite parameterisation used by pPXF struggles. Panels 2 and 5 lie along the disk plane, where the PNKR reconstruction is affected by over-smoothing.
    }
    \label{compare_velocity_distributions}
\end{figure}

We fit the noisy, mock data $y$ testing several settings. Using pPXF we individually fit each of the $625$ spectra in the datacube varying over three settings: the order of Gauss-Hermite expansion (6 or 8), the amount of regularisation for stellar-populations (5 logarithmically spaced values from  $10^{-4}$ to  $10^{4}$), and the amount of regularisation for kinematics (5 logarithmically spaced values from $10^{-2}$ to $10^{2}$ times the pPXF default). We then fit the whole datacube using PNKR. We try both piecewise constant ($s=0$) and piecewise linear ($s=1$) basis functions, and for the latter we also vary the $\beta$ parameter defined in equation \eqref{beta} which influences the smoothness of the recovery. Figures~\ref{compare_moments} and \ref{compare_velocity_distributions} compare the recoveries using pPXF and PNKR. We show results using the most favourable input settings in each case. For pPXF this is an order 6 expansion, population regularisation parameter of $10^4$, and the default kinematic regularisation. For PNKR, this is $s=1$ and $\beta=1$.

Overall, PNKR recoveries are comparable to those achieved with pPXF. The two left-most columns of Figure~\ref{compare_moments} show that neither approach well reproduces the mean age or metallicity maps, likely due to a well known degeneracy between these parameters \cite{Worthey94}. Both methods recover hints of a more metal-rich disk while only pPXF recovers the presence of younger stars closer to the disk-plane. The uniformity of the mean age and metallicity maps recovered by PNKR suggests that they have been over-smoothed. This problem may be alleviated by appropriately tuning the sequence $\bv$ in \eqref{beta_vector}. For the initial tests presented here, we have limited ourselves to vary a scalar tuning-parameter $\beta$ as in \eqref{beta}. The effect of varying this are shown in Figure \ref{compare_h5}. Tuning the full sequence $\bv$ instead would alter the balance between age and metallicity components relative to velocity. This could enhance age and metallicity variations in the PNKR recoveries. Given the intrinsic degeneracy in age and metallicity recoveries (as demonstrated by the pPXF) for now we instead focus our comparison on the other components of the recovered $f$.

Kinematics reveal a significant difference between the quality of the two recoveries. The four right-most columns of Figure~\ref{compare_moments} show that kinematics are overall more faithfully recovered when using PNKR than pPXF. The reason for this can be seen in Figure \ref{compare_velocity_distributions}, where it becomes clear that pPXF's Gauss-Hermite expansions struggle to describe bimodal distributions (e.g. panels 4 and 6 of Figure \ref{compare_velocity_distributions}). These shortcomings where not fixed by using higher-order expansions (up to $h_8$) or less kinematic regularisation, both of which boosted the bimodal features at the expense of severe, negative oscillations in the wings. Other spectral modelling tools \cite{Ocvirk06,bayeslosvd} permit more suitable descriptions of the velocity distribution, while even pPXF itself offers an option to model the spectrum as a mixture of two independent components, which would very likely be beneficial in this instance. We stress however that the purpose of these comparisons is an initial (successful) validation of PNKR; the novelty of PNKR will be demonstrated below. 

Whilst PNKR achieves more accurate kinematic reconstructions overall, the situation is reversed in some locations. PNKR struggled to reconstruct moments $\mu_v, \sigma_v$, and $h_3$ along the disk plane where pPXF fares relatively well (see the central horizontal stripe in bottom two rows of Figure~\ref{compare_moments}, and panels 2 and 5 of Figure \ref{compare_velocity_distributions}). This appears to be a problem with over-smoothing in PNKR: spatial smoothing has been too strong in this region where the thin disk causes kinemtics to change over a small distance. This may be addressed in future work by making a component (AKA mixture-model) ansatz on the density $f$, and separately controlling the smoothing of each component. Furthermore, PNKR reconstructions show some artifacts around the edges (see the third row of Figure~\ref{compare_moments}) caused by the application of the matrix $\Mv^{-1}$ in the case of $s=1$, representing the adjoint embedding $E_s^*$. This could be addressed by artificially enlarging the domain $\Omega$.

\begin{figure}[ht!]
    \centering
    \includegraphics[width=\textwidth]{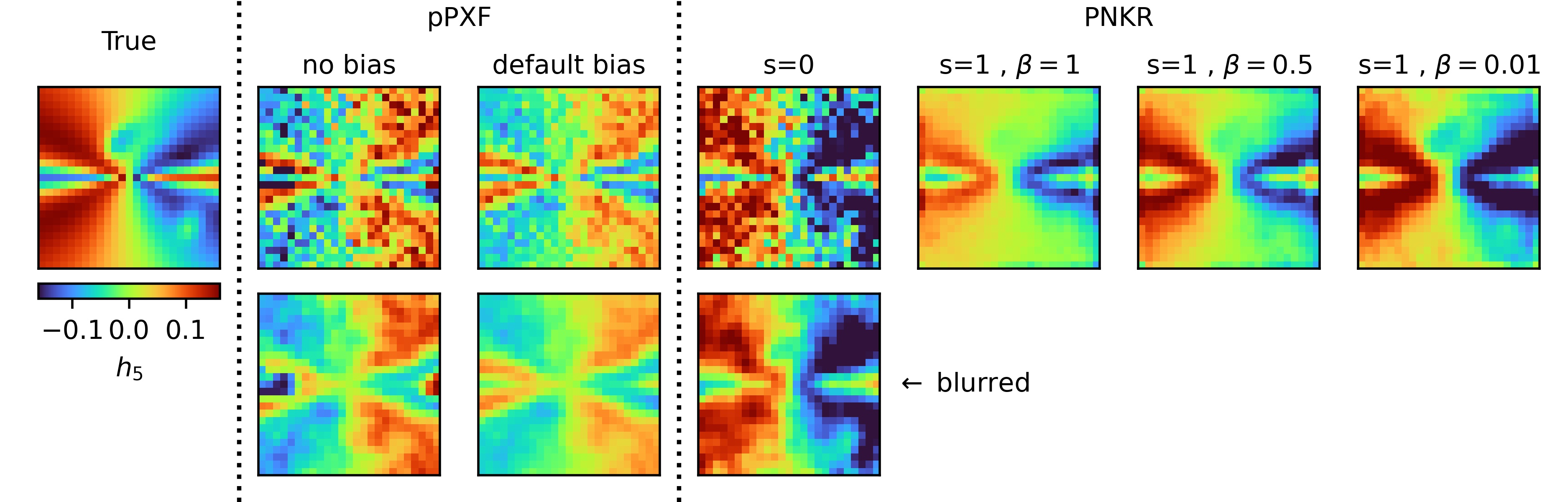}
    \caption{
    Comparison of recovered $h_5$ maps. Left/middle/right sections show the ground-truth/pPXF/PNKR reconstructions. The bottom row shows smoothed versions (blurred with a 0.75 pixel Gaussian kernel) of non-smooth recoveries. The two blue/green blobs in the ground truth are produced by a stellar stream contributing 1\% of the galaxy's stellar mass. These blobs are are not recovered by pPXF regardless of kinematic reguarisation (AKA bias) or smoothing. For PNKR, blobs are visible in the smoothed version of the piecewise-constant ($s=0$) recovery, however the image is somewhat noisy. For the piecewise-linear ($s=1$) case, after tuning the $\beta$ parameter to avoid over-smoothing, the blobs are recovered more prominently and more accurately.
    }
    \label{compare_h5}
\end{figure}

The main success of PNKR is demonstrated in Figure \ref{compare_h5}, which shows maps of the Gauss-Hermite coefficient $h_5$. The faint stellar stream is particularly prominent in the $h_5$ map, appearing as the two blue/green blobs in the left-most panel. In the pPXF recoveries, while the presence of two counter-rotating disks can still be seen inferred, the sign of $h_5$ is flipped relative to the ground truth and no blobs are visible. These conclusions do not change when we turn off kinematic regularisation (AKA bias in pPXF), or when we smooth the reconstructions (bottom row). The PNKR recoveries are significantly better. For the piecewise constant (i.e. $s=0$) case, the blobs are visible in a smoothed version of the recovery, however this image is overall more \emph{blobby}. By using piecewise linear basis functions (i.e. $s=1$) instead, the recovery is smooth by design. The first attempt with $\beta=1$ was overly smooth, but by reducing to $\beta=0.01$ we recover the stream features once again. Strategies for tuning $\beta$ will be explored in a future work. Compared to the smoothed $s=0$ image, in the $(s,\beta)=(1,0.01)$ recovery the blobs appear more cleanly against a smoother background. Furthermore, they are recovered more accurately in their colour and shape. This demonstrates the unique power of PNKR: by modelling the full datacube in a flexible, non-parametric way, weak galactic components can be detected and more accurately characterised.

\subsection{Robustness of the results}\label{ssec_robustness}

Lastly, we assess the robustness of our recoveries. To do this, we create a sample of 15 versions of the mock datacube with different noise realisations, each with an overall noise level of 1\%. We compute PNKR recoveries on the sample, using piecewise-liner basis functions and $\beta=1$. We calculate the median absolute error of the recovered velocity distributions (i.e. those shown in Figure \ref{compare_velocity_distributions}). The mean of this error over positions is 25\%. Quantifying this error is valuable for the astronomical community, to facilitate subsequent analyses (e.g. dynamical modelling) which require a given precision. To provide a useful guide for astronomers, in follow-up work we will quantify this error as a function of the noise level of the datacube.

\section{Conclusion}\label{sect_conclusion}

In this paper, we considered the extragalactic archaeology problem of reconstructing a galaxy's stellar population-kinematic distribution function from measured optical IFU datacubes. Mathematically, this amounts to solving a large scale inverse problem corresponding to an integral equation with a system structure. For solving this problem, we proposed the Projected Nesterov-Kaczmarz Reconstruction (PNKR) method, which incorporates physical prior information and leverages the system structure for numerical efficiency. We tested the PNKR method on a dataset simulated from a known ground truth density, and validated it by comparing our recoveries to those obtained by the widely used pPXF method. We have demonstrated the novel power of full datacube modelling to detect and characterise weak galactic features.

Our results also inspire the following future work: We will explore strategies for tuning the smoothing parameter $\bv$ and treating edge-artifacts. Furthermore we will investigate the possibility to incorporate additional prior information, such as a component ansatz on the distribution function, and treat other known systematics in IFU datasets such as continuum offsets from dust-attenuation and imperfect flux-calibration. Altogether, this is a promising new avenue for extragalactic archaeology; we will apply this method to galaxy IFU data to uncover weak features, such as merger remnants.


\section{Support}

F.\ Hinterer, S.\ Hubmer, and R.\ Ramlau were funded by the Austrian Science Fund (FWF): F6805-N36. P.\ Jethwa and G.\ van de Ven were funded by the Austrian Science Fund (FWF): F6811-N36. P.\ Jethwa and G.\ van de Ven acknowledge funding from the European Research Council (ERC) under the European Union's Horizon 2020 research and innovation programme under grant agreement No 724857 (Consolidator Grant ArcheoDyn). 

\bibliographystyle{plain}
{\footnotesize
\bibliography{mybib}
}

\end{document}

%% file: header.tex
\usepackage{amsmath}
\usepackage{amsthm}
\usepackage{amssymb}
\usepackage{cite}
\usepackage{url}
\usepackage{footmisc}
\usepackage{graphicx}
\usepackage{epstopdf}
\usepackage{chngcntr}
\usepackage{enumitem}
\usepackage{hhline}
\usepackage{array}
\usepackage{hyperref}
\usepackage{color}
\usepackage{algorithm}
\usepackage{algpseudocode}

\hypersetup{colorlinks=false}

\if@twoside
\else 
\fi

\numberwithin{equation}{section}
\counterwithin{figure}{section}
\counterwithin{table}{section}

\theoremstyle{plain}
\newtheorem{theorem}{Theorem}[section]
\newtheorem{lemma}[theorem]{Lemma}
\newtheorem{proposition}[theorem]{Proposition}

\newtheorem{problem}{Problem}

\theoremstyle{definition}
\newtheorem{definition}{Definition}[section]
\newtheorem{assumption}{Assumption}[section]

\newtheorem*{PNKR}{The PNKR method}
\newtheorem*{SPNKR}{The reduced PNKR method}

\theoremstyle{remark}

\newcommand{\norm}[1]{\left\|#1\right\|}
\newcommand{\abs}[1]{\left\vert#1\right\vert}
\newcommand{\spr}[1]{\left\langle\,#1\,\right\rangle}
\newcommand{\kl}[1]{\left(#1\right)}
\newcommand{\Kl}[1]{\left\{#1\right\}}

\newcommand{\R}{\mathbb{R}} 

\newcommand{\N}{\mathbb{N}}

\newcommand{\yd}{y^{\delta}}

\newcommand{\s}{\text{s}}

\newcommand{\okd}{\omega_k^\delta}
\newcommand{\okdt}{\Tilde{\omega}_k^\delta}
\newcommand{\akd}{\alpha_k^\delta}

\newcommand{\fkd}{f_k^\delta}
\newcommand{\fkpd}{f_{k+1}^\delta}
\newcommand{\uvkd}{\uv_k^\delta}
\newcommand{\uvkpd}{\uv_{k+1}^\delta}

\newcommand{\fkdt}{\tilde{f}_{k}^\delta}
\newcommand{\fkpdt}{\tilde{f}_{k+1}^\delta}

\newcommand{\uikd}{u_{i,k}^\delta}
\newcommand{\uikpd}{u_{i,k+1}^\delta}

\newcommand{\xv}{\boldsymbol{x}}
\newcommand{\yv}{\boldsymbol{y}}
\newcommand{\yvd}{\boldsymbol{y}^\delta}
\newcommand{\Kv}{\boldsymbol{K}}

\newcommand{\vmin}{{v_\text{min}}}
\newcommand{\vmax}{{v_\text{max}}}
\newcommand{\zmin}{{z_\text{min}}}
\newcommand{\zmax}{{z_\text{max}}}
\newcommand{\tmin}{{t_\text{min}}}
\newcommand{\tmax}{{t_\text{max}}}
\newcommand{\lmin}{{\lambda_\text{min}}}
\newcommand{\lmax}{{\lambda_\text{max}}}

\newcommand{\LtOT}{{L^2\kl{\Omega \times \Theta}}}
\newcommand{\HsOT}{{H^s\kl{\Omega \times \Theta}}}
\newcommand{\LtOL}{{L^2\kl{\Omega \times \Lambda}}}
\newcommand{\LtO}{{L^2\kl{\Omega}}}
\newcommand{\LtTL}{{L^2\kl{\Theta\times\Lambda}}}

\newcommand{\kmod}{{[k]}}
\newcommand{\Kt}{\tilde{K}}

\newcommand{\Mv}{\boldsymbol{M}}
\newcommand{\Nv}{\boldsymbol{N}}
\newcommand{\Hv}{\boldsymbol{H}}
\newcommand{\Iv}{\boldsymbol{I}}
\newcommand{\uv}{\boldsymbol{u}}
\newcommand{\wv}{\boldsymbol{w}}
\newcommand{\zv}{\boldsymbol{z}}

\newcommand{\XM}{{X_M}}
\newcommand{\YN}{{Y_N}}
\newcommand{\bv}{{\boldsymbol{\beta}}}
